\documentclass{article}
\usepackage{amsmath} % Package g\'en\'eral pour \'ecrire des maths...
\usepackage{mathtools} % Vient compl\'eter amsmath avec davantage d'outils.
\usepackage{amssymb} % Permet l'utilisation de certains symboles (ex, le R de l'ensemble des r\'eels); cf documentation.
\usepackage{stmaryrd} % St Mary Road's symbols : vient compl\'eter et am\'eliorer amssymb, notamment \llbracket et \rrbracket; cf documentation
\usepackage{mathrsfs} % Ralph Smith Formal Script ; nouveaux caract\`eres calligraphiques avec \mathscr{ABC}
\usepackage{amsthm}
\usepackage{microtype}

\usepackage[margin=3cm]{geometry}
\usepackage[shortlabels]{enumitem}

\usepackage{todonotes}
\usepackage{latexsym}
\usepackage{xspace}
\usepackage{booktabs}

\usepackage{graphicx} % Plus complet que "graphics", offre des options lors de l'ajout d'objets avec \includegraphics
\usepackage{float} % Facilite l'ajout d'objets float (figures, tables...)
\usepackage{caption} % Offre plein d'options sp\'ecifiques pour l'ajout de sous-titres \`a des objets float.
\usepackage{subcaption} % permet l'ajout de sous-titres \`a des subfigures.
\usepackage{siunitx}
\usepackage{tikz}
\usepackage{tikz-cd}
\usepackage{blindtext} % Pour rapidement ajouter des pav\'e de texte sans sens particulier
\usepackage{comment}
\usepackage{hyperref}
\usepackage{csquotes}

\usepackage{titlesec}

\titleformat{\subsubsection}[runin]
  {\normalfont\normalsize\itshape}
  {\thesubsubsection}
  {0.6em}
  {}
  [. ]

\titlespacing*{\subsubsection}
  {0pt}{1.25ex}{0.5em}

\newcommand{\vect}[1]{\text{Vect}}

\newcommand{\R}{\mathbb{R}}

\newcommand{\N}{\mathbb{N}}
\newcommand{\Z}{\mathbb{Z}}
\newcommand{\E}{\mathbb{E}}

\newcommand{\floor}[1]{\lfloor #1 \rfloor}

\newcommand{\scalp}[2]{\langle #1 \, , #2\rangle}

\newcommand{\Obs}{\mathcal{O}}

\newcommand{\End}{\mathrm{End}}
\newcommand{\Hom}{\mathcal{L}}
\newcommand{\tr}{\mathrm{Tr}}
\newcommand{\supp}{\mathrm{supp}\,}
\newcommand{\rank}{\mathrm{rank}}

\newcommand{\HSnorm}[1]{\left|#1\right|_{\mathrm{HS}}}

\newcommand{\V}{\mathcal{V}}
\newcommand{\Vbox}{\V\boxtimes \V^*}
\newcommand{\Mloc}{\mathcal{M}_{\mathrm{loc}}}
\newcommand{\W}{\mathcal{S}(\R)}

\newtheorem{thm}{Theorem}
\newtheorem{prop}{Proposition}
\newtheorem{cor}{Corollary}
\newtheorem{lem}{Lemma}

\newtheorem{defn}{Definition}
\newtheorem{remark}[thm]{Remark}

\title{Optimal geometric barriers for weighted observability of heat semigroups on metric measure spaces}
\author{Vincent Boulard\footnote{CERMICS, CNRS, \'Ecole nationale des ponts et chauss\'ees, Institut Polytechnique de Paris, Marne-la-Vall\'ee, France (\texttt{vincent.boulard@enpc.fr}).}, Amaury Hayat\footnote{CERMICS, CNRS, \'Ecole nationale des ponts et chauss\'ees, Institut Polytechnique de Paris, Marne-la-Vall\'ee, France (\texttt{amaury.hayat@enpc.fr}).}, and Emmanuel Tr\'elat\footnote{Sorbonne Universit\'e, Universit\'e Paris Cit\'e, CNRS, Inria, Laboratoire Jacques-Louis Lions, LJLL, F-75005 Paris, France (\texttt{emmanuel.trelat@sorbonne-universite.fr}).}}
\date{}

\begin{document}

\maketitle

\begin{abstract}
Weighted integrated observability inequalities for heat equations usually involve a small-time factor of the form $e^{-\gamma/t}$. We prove that this scale is not an artefact of Carleman or spectral methods: it is forced by the geometry of the observation set.

Let $A$ be a nonnegative self-adjoint operator on sections of a finite-rank Euclidean vector bundle over a doubling metric measure space, satisfying ultracontractivity, Davies-Gaffney estimates (equivalently, finite speed of propagation for the wave equation) and a pointwise local Weyl law. If a weighted integrated observability inequality holds on a measurable set $\omega$, for a fixed horizon $T\in(0,+\infty]$ and an admissible weight $h$, then, for every $0<\kappa<\frac{1}{2}$,
$$
h(t)\leq A_{T,\kappa}\exp\left(-\kappa\frac{\mathcal{L}(\omega)^2}{t}\right),\qquad 0<t<T,
$$
where $\mathcal{L}(\omega)$ is the essential maximal distance to $\omega$, replaced by any finite radius when $\mathcal{L}(\omega)=+\infty$. Thus, for $h(t)=e^{-\gamma/t}$, necessarily $\gamma\geq\mathcal{L}(\omega)^2/2$. This settles, with the optimal threshold, the maximal-distance lower bound for the infinite-time constant left open in earlier work. In the control-norm convention, the fast-control rate is at least $\mathcal{L}(\omega)^2/4$, recovering Miller's bound.

The proof rests on the spectral packet $(\cosh(r\sqrt A)-1)e^{-tA}$. A pointwise Weyl law gives its sharp lower growth, while finite propagation speed and a weak-kernel Kannai transmutation formula make it exponentially small on $\omega$. Without kernel continuity or compact resolvent, we develop pointwise spectral measures and weak wave kernels. The framework covers Laplace-type operators on compact Riemannian manifolds, coupled heat systems, Schr\"odinger operators on $\mathbb{R}^d$, equiregular sub-Laplacians and Grushin models, and $\delta'$-coupled Laplacians on metric graphs.
\end{abstract}

\tableofcontents

\section{Introduction and main theorem}\label{sec:intro}
%%%%%%%%%%%%%%%%%%%%%%%%%%%%%%%%%%%%%

\subsection{Context and aims}\label{subsec:context}

A recurrent theme in the controllability and observability literature for heat equations is the appearance of an exponential weight $e^{-\gamma/t}$ in small time. Let $(M,g)$ be a smooth connected compact $d$-dimensional Riemannian manifold (with Dirichlet conditions if $\partial M\neq\emptyset$), and let $\omega\subset M$ be a nonempty open subset with $\overline{\omega} \neq M$. Denote by $\Delta_g$ the nonnegative Laplace--Beltrami operator on $L^2(M)=L^2(M,\mu)$, where $\mu$ is the Riemannian measure. It was shown in \cite{LebeauRobbiano1995CPDE} (by means of a spectral inequality proved via local elliptic Carleman estimates) and in \cite{FursikovImanuvilov1996,Imanuvilov1995Sbornik} (by means of global parabolic Carleman estimates) that the heat equation is observable in $\omega$ for any finite time $T>0$: there exists $C_T>0$ such that
\begin{equation}\label{finite-obs}
\forall f\in L^2(M),\qquad \|e^{-T\Delta_g}f\|_{L^2(M)}^2 \leq C_T \int_0^T \|e^{-t\Delta_g}f\|_{L^2(\omega)}^2\,dt.
\end{equation}
It is well known (see \cite{FursikovImanuvilov1996,LeRousseauLebeau2012COCV,LebeauRobbiano1995CPDE}) that the optimal constant in \eqref{finite-obs} blows up exponentially as $T\to 0$.
More precisely, there exist constants $C,K_0>0$ such that, for all $T\in(0,1]$, the optimal constant $C_T$ in \eqref{finite-obs} satisfies
\begin{equation}\label{opti-K}
    C_T \leq C e^{K_0/T},
\end{equation}
and this exponential growth is genuine: a matching exponential lower bound on $C_T$ as $T\to0$ also holds (see \cite{Guichal1985_ControlOperatorHeat} in the one-dimensional case and the introduction of \cite{LaurentLeautaud2021APDE} in general). We denote by $K_*$ the infimum of all exponents $K_0$ for which \eqref{opti-K} holds; thus $K_*$ is the small-time exponential rate of the \emph{squared} observability constant. By the classical duality of \cite{DoleckiRussell1977SICON} (see also \cite{Lions1988SIAMRev,Trelat2024CFID}), \eqref{finite-obs} is equivalent to exact null-controllability in time $T$ for $\partial_t y+\Delta_g y=1_\omega u$. If $\mathcal{C}_T$ denotes the usual $L^2$ null-control cost, then $C_T=\mathcal{C}_T^2$. Consequently, if $\mathcal{K}_{\rm heat}:=\inf\{K>0\,\mid\, \mathcal{C}_T\lesssim e^{K/T}\text{ as }T\to0^+\}$, then $K_*=2\mathcal{K}_{\rm heat}$. We use the squared-observability normalization throughout the paper.

A fundamental question underlies the appearance of this small-time exponential weight: \emph{is the scale $e^{-\gamma/t}$ forced by the geometry, or is it merely a by-product of the Carleman method?} The main contribution of this article is to answer this question, in a very general abstract metric measure space setting, by proving that the exponential behavior is an intrinsic geometric barrier. The resulting obstruction has the same logarithmic form as Varadhan's short-time heat-kernel asymptotics.

Beyond \eqref{finite-obs}, two integrated formulations of observability play a central role:
\begin{enumerate}[(i)]
\item \emph{Finite-time integrated observability} \cite{FernandezCaraZuazua2000ADE,Miller2006RLM}: there exist $\tilde\gamma>0$ and $\tilde C>0$ such that for every $T>0$,
\begin{equation}\label{intobs}
\forall f\in L^2(M),\quad \int_0^T e^{-\tilde\gamma/t}\|e^{-t\Delta_g}f\|_{L^2(M)}^2\,dt \leq \tilde C\int_0^T \|e^{-t\Delta_g}f\|_{L^2(\omega)}^2\,dt.
\end{equation}
\item \emph{Infinite-time integrated observability} \cite{Zuazua2001Carleman}: there exist $\tilde\gamma_\infty>0$ and $C_\infty>0$ such that
\begin{equation}\label{infintobs}
\forall f\in(\ker\Delta_g)^\perp,\quad \int_0^{+\infty} e^{-\tilde\gamma_\infty/t}\|e^{-t\Delta_g}f\|_{L^2(M)}^2\,dt \leq C_\infty\int_0^{+\infty} \|e^{-t\Delta_g}f\|_{L^2(\omega)}^2\,dt.
\end{equation}
For the Dirichlet Laplacian, $\ker\Delta_g=\{0\}$; on a closed connected manifold, $(\ker\Delta_g)^\perp$ is the zero-mean subspace. This restriction is necessary because the two time integrals diverge on nonzero stationary states.
\end{enumerate}
A key result of \cite{Miller2006RLM} is that (i) holds if and only if the finite-time observability \eqref{finite-obs} holds for every $T>0$, and if we denote by $\gamma$ (by $\gamma_\infty$) the optimal exponential constant in \eqref{intobs} (in \eqref{infintobs}), then $\gamma=K_*$. Equivalently, in the usual control-norm convention, $\gamma=2\mathcal{K}_{\rm heat}$. Moreover, \eqref{intobs} implies \eqref{infintobs} with $\gamma\geq\gamma_\infty$.
We shall use the standard fact that the small-time formulation can be converted into the above all-time formulation. A precise statement is recalled in Appendix~\ref{subsec:small-large}.

In \cite{FernandezCaraZuazua2000ADE}, the authors also noticed that \eqref{intobs} implies, via Laplace's method (see \cite[Chapter 4]{deBruijn1958}), the following \emph{sharp observability inequality}: there exists $c_0>0$ such that, for all $T>0$, there exist $K'',C'>0$ with
\begin{equation}\label{sharpobs}
\forall f\in L^2(M),\qquad \|e^{-c_0\sqrt{\Delta_g}}f\|_{L^2(M)}^2 \leq C'e^{K''/T}\int_0^T \|e^{-t\Delta_g}f\|_{L^2(\omega)}^2\,dt.
\end{equation}

The optimal exponential constants $\gamma$ and $\gamma_\infty$ are geometric quantities. The distance relevant to the present result is
\begin{equation*}
\mathcal{L}(\omega)=\sup_{x\in M}d(x,\omega).
\end{equation*}
Earlier Carleman arguments yield lower bounds for $\gamma$ and $\gamma_\infty$ in terms of a boundary-sensitive inradius-type quantity $D(\omega)\leq\mathcal{L}(\omega)$; we refer to \cite{Zuazua2001Carleman} for its precise definition and for the estimate $\gamma,\gamma_\infty\geq D(\omega)^2/2$. Miller's heat-kernel argument \cite{Miller2004JDE} gives the stronger maximal-distance bound $\gamma\geq\mathcal{L}(\omega)^2/2$ for the finite-time constant, equivalently $\mathcal{K}_{\rm heat}\geq\mathcal{L}(\omega)^2/4$ in the control-norm convention. Whether the corresponding maximal-distance bound also holds for $\gamma_\infty$ was left open in \cite[Section~5]{ErvedozaZuazua2011}. Corollary~\ref{cor:conj_ervzua} answers this question. For a detailed analysis of the different geometric constants and of the limitations of a universal distance-only formula for the full control cost, see \cite{LaurentLeautaud2021APDE}.

We refer to Subsection~\ref{subsec:discussion} for a discussion of the novelty of our approach and of its relation to previous results.

\subsection{Main results}\label{subsec:main-results}

We work on a locally compact, separable metric measure space $(X,d,\mu)$ with $\supp\mu=X$ and the doubling volume property, and consider a nonnegative self-adjoint operator $A$ on $L^2(X;\V)$ (sections of a Euclidean vector bundle $\V\to X$) satisfying three structural hypotheses stated precisely in Section~\ref{sec:framework}: ultracontractivity~\ref{item:A1}, Davies--Gaffney estimates~\ref{item:A2}, and a local Weyl law~\ref{item:A3}. We fix a non-negligible measurable observation set $\omega\subset X$ and a time horizon $T\in(0,+\infty]$.

Since the observability inequality below is insensitive to modifications of $\omega$ on a $\mu$-null set, the geometric quantity governing the barrier must be defined accordingly. We therefore use the \emph{essential distance} to $\omega$,
\begin{equation}\label{ess-dist}
d_{\mathrm{ess}}(x,\omega) = \sup\{ r > 0 \,\mid\, \mu(B(x,r) \cap \omega) = 0 \}, \qquad \mathcal{L}(\omega) = \sup_{x\in X} d_{\mathrm{ess}}(x,\omega) \in[0,+\infty],
\end{equation}
with the convention $\sup\emptyset = 0$. This makes $\mathcal{L}(\omega)$ insensitive to $\mu$-null modifications of $\omega$ and does not require $\omega$ to be open. When $\omega$ is open it coincides with $\sup_{x}d(x,\omega)$. The only properties used below are that, for every $\delta\geq 0$, the superlevel set $U_\delta = \{ x \,\mid\, d_{\mathrm{ess}}(x,\omega) > \delta \}$ is open (the function $d_{\mathrm{ess}}(\cdot,\omega)$ being $1$-Lipschitz) and, whenever nonempty, of positive measure (since $\supp\mu = X$), and that $d(x,y)\geq d_{\mathrm{ess}}(x,\omega)$ for $\mu$-a.e.\ $y\in\omega$.

We define the set of admissible weights by
\[
\Obs([0,T])=\left\{h\in C^0_b([0,T))\ \mid\ \begin{array}{l}
 h\geq0,\ \text{and there exists }\tau_h\in(0,T)\\[0mm]
 \text{such that }h\text{ is nondecreasing on }[0,\tau_h]
\end{array}\right\}.
\]
Here $C^0_b([0,T))$ denotes the space of bounded continuous functions on $[0,T)$.

\begin{defn}\label{defn:h-obs}
Fix $T\in(0,+\infty]$ and $h\in\Obs([0,T])$. We say that an \emph{integrated observability inequality with weight $h$} holds if
\begin{equation}\label{gen-obs}
\forall f\in L^2(X;\V),\quad \int_0^T h(t)\|e^{-tA}f\|_{L^2(X;\V)}^2\,dt \leq \int_0^T \|e^{-tA}f\|_{L^2(\omega;\V)}^2\,dt.
\end{equation}
When $T=+\infty$, the inequality is understood for every $f$ for which the right-hand side is finite; this excludes vacuous comparisons of the form $+\infty\leq+\infty$.
\end{defn}

The normalization with constant one on the right-hand side is harmless: an inequality with a constant $C_{\mathrm{obs}}$ is reduced to \eqref{gen-obs} by replacing $h$ with $h/C_{\mathrm{obs}}$, which does not change any small-time exponential rate; this is what makes the passage to Corollary~\ref{cor:conj_ervzua} legitimate.

All conclusions below remain valid if \eqref{gen-obs} is assumed only for $f\in(\ker A)^\perp$. Indeed, every test section used in the proof belongs to $(\ker A)^\perp$, because its spectral multiplier contains the factor $\cosh(r\sqrt\lambda)-1$, which vanishes at $\lambda=0$. We use this variant for infinite-time observability.

\begin{thm}[Geometric barrier]\label{thm:main_thm}
Fix $T\in(0,+\infty]$ and $h\in\Obs([0,T])$ such that \eqref{gen-obs} holds.
If $\mathcal{L}(\omega)<+\infty$, then, for every $0<\kappa<\frac{1}{2}$, there exists $A_{T,\kappa}>0$ such that
\begin{equation}\label{main-barrier}
\forall t\in(0,T),\qquad h(t)\leq A_{T,\kappa}\,e^{-\kappa\mathcal{L}(\omega)^2/t}.
\end{equation}
If $\mathcal{L}(\omega)=+\infty$, then, for every $R>0$ and every $0<\kappa<\frac{1}{2}$, there exists $A_{T,R,\kappa}>0$ such that
\[
\forall t\in(0,T),\qquad h(t)\leq A_{T,R,\kappa}\,e^{-\kappa R^2/t}.
\]
Equivalently, with the conventions $\log0=-\infty$ and $-\mathcal{L}(\omega)^2/2=-\infty$ when $\mathcal{L}(\omega)=+\infty$, these estimates are summarized by the sharp logarithmic endpoint
\begin{equation}\label{main-varadhan}
\limsup_{t\downarrow0}t\log h(t)\leq-\frac{1}{2}\mathcal{L}(\omega)^2.
\end{equation}
The constants are allowed to depend on $h$, $A$, $\omega$ and the displayed parameters; no uniformity as $\kappa\uparrow\frac{1}{2}$ is asserted.
\end{thm}

\begin{cor}[Infinite-time constant lower bound]\label{cor:conj_ervzua}
Let $\tilde\gamma_\infty>0$. Assume that there exists $C_\infty>0$ such that, for every $f\in(\ker A)^\perp$ for which the right-hand side is finite,
\[
\int_0^{+\infty} e^{-\tilde\gamma_\infty/t}\|e^{-tA}f\|_{L^2(X;\V)}^2\,dt
\leq C_\infty\int_0^{+\infty}\|e^{-tA}f\|_{L^2(\omega;\V)}^2\,dt.
\]
Then $\tilde\gamma_\infty\geq\frac{1}{2}\mathcal{L}(\omega)^2$. In particular, if $\mathcal{L}(\omega)=+\infty$, no finite $\tilde\gamma_\infty$ can satisfy such an inequality.
\end{cor}

In Subsection~\ref{subsec:riem-app}, we verify the assumptions for the Dirichlet Laplacian on a compact connected Riemannian manifold with boundary. Corollary~\ref{cor:conj_ervzua} therefore answers the open problem raised in \cite{ErvedozaZuazua2011,Zuazua2001Carleman}: the infinite-time integrated constant satisfies $\gamma_\infty\geq\mathcal{L}(\omega)^2/2$. The coefficient $1/2$ is optimal as a universal lower-bound constant, as shown by the sharp one-dimensional and radially symmetric model configurations of \cite{ErvedozaZuazua2011}; no constant-prefactor estimate at the endpoint $\kappa=1/2$ is asserted for a general weight.

The same barrier yields a lower bound on the exponential rate of fast observability. The next statement is written in the squared-observability normalization of \eqref{finite-obs}; in the usual control-norm normalization, its constant is divided by two.

\begin{cor}[Cost of fast controls]\label{cor:fast-controls}
Assume that there exist $T_0>0$, $C>0$ and $K>0$ such that, for every $T\in(0,T_0]$,
\begin{equation}\label{fast-obs}
\forall f\in L^2(X;\V),\qquad \|e^{-TA}f\|_{L^2(X;\V)}^2 \leq C\, e^{K/T}\int_0^T\|e^{-tA}f\|_{L^2(\omega;\V)}^2\,dt.
\end{equation}
Then $K\geq \frac{1}{2}\mathcal{L}(\omega)^2$. Equivalently, the usual $L^2$ null-control rate $\mathcal{K}_{\rm heat}$ satisfies $\mathcal{K}_{\rm heat}\geq\mathcal{L}(\omega)^2/4$.
\end{cor}

On a smooth compact Riemannian manifold, this is precisely Miller's lower bound \cite{Miller2004JDE}, after converting between the control norm and the squared observability constant. The abstract statement extends the same obstruction, for instance, to the sub-Riemannian and metric-graph settings covered below.

\medskip
Now, denote by $E(\cdot)$ the spectral measure of $A$. For every $h\in\Obs([0,T])$, we have
\[\int_0^T h(t)\|e^{-tA}f\|_{L^2(X;\V)}^2\,dt = \int_0^\infty\int_0^T h(t)e^{-2t\lambda}\,dt\,d\scalp{E(\lambda)f}{f}_{L^2(X;\V)}.\]

\begin{defn}\label{defn:HT}
Let $T\in(0,+\infty)$ and $h\in\Obs([0,T])$. We define
\[\forall\lambda\geq 0,\quad H_T(\lambda)=\sqrt{\int_0^T h(t)e^{-2t\lambda}\,dt},\]
which is bounded and continuous on $[0,+\infty)$.
\end{defn}

If \eqref{gen-obs} holds, then $\|H_T(A)f\|_{L^2(X;\V)}^2\leq\int_0^T\|e^{-tA}f\|_{L^2(\omega;\V)}^2\,dt$, providing a passage from integrated to ``final-time'' observability. As noticed in \cite{FernandezCaraZuazua2000ADE}, only the high frequencies contribute (via Laplace's method, see \cite[Chapter 4]{deBruijn1958}).

\begin{cor}[High-frequency constraint]\label{cor:obs_sharp_cor}
Let $T\in(0,+\infty)$ and $h\in\Obs([0,T])$ satisfy \eqref{gen-obs}. If $\mathcal{L}(\omega)<+\infty$, then, for every $0<\beta<1$, there exists $A'_{T,\beta}>0$ such that
\begin{equation}\label{H-sharp}
\forall\lambda\geq0,\qquad H_T(\lambda)\leq A'_{T,\beta}\,e^{-\beta\mathcal{L}(\omega)\sqrt\lambda}.
\end{equation}
If $\mathcal{L}(\omega)=+\infty$, then, for every $R>0$ and every $0<\beta<1$, there exists $A'_{T,R,\beta}>0$ such that $H_T(\lambda)\leq A'_{T,R,\beta}e^{-\beta R\sqrt\lambda}$ for all $\lambda\geq0$.
\end{cor}

This shows that the sharp observability inequality \eqref{sharpobs} is optimal within the integrated approach. In this sense, the integrated approach cannot yield a high-frequency weight with a better exponential scale than $e^{-c\mathcal{L}(\omega)\sqrt\lambda}$. It does not, however, establish the optimality of the constants appearing in a Lebeau--Robbiano type inequality.

\begin{remark}\label{rem:infty-rem}
If $\mathcal{L}(\omega)=+\infty$, \eqref{main-varadhan} reads $\limsup_{t\downarrow0}t\log h(t)=-\infty$. Equivalently, Theorem~\ref{thm:main_thm} gives the finite-radius estimate $h(t)\leq A_{T,R,\kappa}e^{-\kappa R^2/t}$ for every $R>0$ and every $0<\kappa<1/2$, so $h$ decays faster than every prescribed scale $e^{-a/t}$ as $t\to0^+$. For an infinite horizon, one may define $H_\infty(\lambda)^2=\int_0^{+\infty}h(t)e^{-2t\lambda}\,dt$ for $\lambda>0$. The same argument gives $H_\infty(\lambda)\leq A'_{c,R,\beta}e^{-\beta R\sqrt\lambda}$ for all $\lambda\geq c$, for every fixed $c>0$, $R>0$ and $0<\beta<1$.
\end{remark}

\subsection{Discussion}\label{subsec:discussion}

Let us emphasize what is genuinely new and what is sharp in our approach, and how it relates to previous works.

The earlier lower bounds must be separated according to their formulations. The estimates $\gamma,\gamma_\infty\geq D(\omega)^2/2$ in \cite{Zuazua2001Carleman} concern integrated inequalities with a prescribed exponential weight, whereas Miller's estimate \cite{Miller2004JDE} concerns the small-time family of null-control costs and yields $\mathcal{K}_{\rm heat}\geq\mathcal{L}(\omega)^2/4$ in the usual control-norm convention, equivalently $K_*\geq\mathcal{L}(\omega)^2/2$ in our squared-observability convention.

Theorem~\ref{thm:main_thm} starts instead from one integrated inequality on one fixed, possibly infinite, horizon, with an arbitrary admissible weight $h$, and extracts its necessarily hidden exponential scale. Its contribution is threefold:
\begin{enumerate}
\item the obstruction is formulated for abstract heat semigroups on doubling metric measure spaces and is governed by the essential maximal distance;
\item the weight is not prescribed to be exponential: the geometry itself forces the quantitative barriers \eqref{main-barrier} for every $0<\kappa<1/2$, equivalently the sharp Varadhan-type logarithmic endpoint \eqref{main-varadhan};
\item the argument applies to a single finite- or infinite-horizon inequality, and in particular yields the previously open maximal-distance bound for the infinite-time constant.
\end{enumerate}
Thus the scale $e^{-\gamma/t}$, ubiquitous in parabolic controllability estimates obtained by Carleman or spectral methods \cite{FernandezCaraZuazua2000ADE,FursikovImanuvilov1996,LeRousseauLebeau2012COCV,LebeauRobbiano1995CPDE}, is intrinsic to integrated observability rather than an artefact of those proofs.

The coefficient $1/2$ is the one predicted by the short-time heat-kernel geometry. In the classical smooth scalar setting, Varadhan's formula \cite{Varadhan1967} reads
\[
\lim_{t\downarrow0}4t\log p_t(x,y)=-d(x,y)^2;
\]
after squaring the kernel, the corresponding geometric exponent is $d(x,y)^2/(2t)$. This logarithmic heuristic underlies Miller's smooth-manifold lower bound \cite{Miller2004JDE}. Our proof reaches the same universal threshold from finite propagation and pointwise spectral measures, without assuming positivity of the semigroup, pointwise Gaussian lower bounds, continuity of the kernels, or compact resolvent.

The coefficient $1/2$ cannot be improved in a universal lower bound. The observability estimates of \cite{ErvedozaZuazua2011} are sharp in one-dimensional and multidimensional radially symmetric model configurations; after conversion to the squared-observability normalization used here, these models exhibit the threshold $\mathcal{L}(\omega)^2/2$. Combined with Corollary~\ref{cor:conj_ervzua}, this identifies the optimal universal coefficient, without asserting validity at the endpoint for a general admissible weight. We emphasize, however, that $\mathcal{L}(\omega)^2$ need not encode the full optimal fast-control cost: the paper \cite{LaurentLeautaud2021APDE} exhibits finer geometric and spectral effects that can make this cost much larger. Our result identifies a universal distance barrier, not a universal exact formula for the cost.

\begin{remark}\label{rem:endpoint}
The quantitative and logarithmic formulations contain the same small-time information. More precisely, when $\mathcal{L}(\omega)<+\infty$, the family of estimates \eqref{main-barrier} for every $0<\kappa<1/2$ is equivalent to
\[
\limsup_{t\downarrow0}t\log h(t)\leq-\frac12\mathcal{L}(\omega)^2,
\]
with the convention $\log0=-\infty$. Indeed, \eqref{main-barrier} implies the logarithmic inequality by letting $\kappa\uparrow1/2$; conversely, the logarithmic inequality yields \eqref{main-barrier} for all sufficiently small $t$, and the boundedness of $h$ extends it to the whole interval $(0,T)$ after enlarging the constant. This is the natural level of comparison with Varadhan's heat-kernel asymptotics \cite{Grigoryan2009_HeatKernel,Varadhan1967}. For the scalar Laplace--Beltrami heat kernel on a closed smooth Riemannian manifold, global two-sided small-time estimates are commonly written, for every $\varepsilon>0$ and up to multiplicative constants, as
\[
c_\varepsilon\, t^{-d/2}\exp\!\left(-\frac{d(x,y)^2}{(4-\varepsilon)t}\right) \leq p_t(x,y)\leq C_\varepsilon\, t^{-d/2}\exp\!\left(-\frac{d(x,y)^2}{(4+\varepsilon)t}\right).
\]
After squaring, the critical denominator is $2$ (with $2-\varepsilon$ in the lower estimate and $2+\varepsilon$ in the upper estimate, after renaming $\varepsilon$), while the displayed polynomial prefactor becomes $t^{-d}$. Thus the restriction $\kappa<1/2$ entails no loss in the geometric exponential rate: it is what allows polynomial, or more generally subexponential, factors to be absorbed. The Varadhan exponential denominator remains the relevant one globally, but an exact two-sided expansion with the standard $t^{-d/2}$ prefactor holds in Euclidean space and locally away from the cut locus; at cut or conjugate points the polynomial prefactor may change (see, e.g., \cite{Ludewig2019StrongShortTime}). Consequently, an estimate of the form $h(t)\leq A e^{-\mathcal{L}(\omega)^2/(2t)}$ with a bounded prefactor is strictly stronger than the logarithmic endpoint proved here. Whether such a constant-prefactor endpoint estimate follows from \eqref{gen-obs} under additional assumptions is a separate question that we do not address.
\end{remark}

The core novelty of the proof lies in the choice of test functions. We insert the localized spectral packet $K_{(\cosh(r\sqrt A)-1)e^{-tA}}(\cdot,x)$ into the integrated inequality and let $r<d_{\mathrm{ess}}(x,\omega)$. Its upper bound on $\omega$ uses finite wave propagation through the classical wave-to-heat Kannai transmutation formula (Lemmas~\ref{resummed-transmutation} and~\ref{resummed-omega}), whereas its lower bound at the base point follows from the local Weyl law (Lemma~\ref{resummed-lower}). Conceptually, this direction of transmutation is complementary to the method of \cite{ErvedozaZuazua2011}, which represents selected wave solutions in terms of heat solutions in order to derive heat observability from wave observability. Here the wave equation is instead used to construct an obstruction to an already assumed heat observability inequality.

\subsection{Organization of the paper}

Section~\ref{sec:framework} introduces the abstract framework: the three hypotheses~\ref{item:A1}--\ref{item:A3} (Subsection~\ref{subsec:hyp}), and the key objects and kernel identities used in the proof of Theorem~\ref{thm:main_thm} (Subsection~\ref{subsec:kernel-id}),
namely the heat kernel, the pointwise spectral measure, and the wave kernel together with the Kannai transmutation formula. Section~\ref{sec:main_proof} contains the proofs of Theorem~\ref{thm:main_thm} and Corollaries~\ref{cor:conj_ervzua}--\ref{cor:obs_sharp_cor}: after a proof outline,
the three key lemmas are stated and proved in Subsection~\ref{subsec:key-lemmas}, and the main results are assembled in Subsection~\ref{subsec:proof_main} and the three subsections that follow it. Section~\ref{sec:applications} verifies Assumptions~\ref{item:A1}--\ref{item:A3} in four geometric settings -- compact Riemannian manifolds and Laplace-type operators, Schr\"odinger operators on $\R^d$, compact sub-Riemannian and Grushin-type structures, and compact metric graphs with $\delta'$-coupling -- and derives the corresponding consequences; a summary table is given at the beginning of the section.

The article is completed by two appendices, which are largely self-contained and may be read as a black box: the reader may take the statements of Section~\ref{sec:framework} for granted and skip directly to the proofs. Appendix~\ref{app:kernel-toolbox} develops the kernel calculus toolbox: the $L^\infty_\Delta$ spaces of integral kernels (Subsection~\ref{app:subsec-Linfty}),
 the heat kernel and the pointwise spectral measure construction under Assumption~\ref{item:A1} (Subsection~\ref{subsec:Ultra-spec}), and the weak wave kernel together with the Kannai transmutation formula and finite speed of propagation
thanks to Assumption~\ref{item:A2} (Subsection~\ref{subsec:wave_kernel}). Appendix~\ref{app:tech-lem} collects the technical lemmas used along the way.

\medskip
\noindent\textbf{Notation.} Throughout, $L^p(X;\V)=L^p(X,\mu;\V)$ denotes the space of $L^p$ sections of $\V$, and $\|\cdot\|_{p\to q}$ is the operator norm from $L^p(X;\V)$ to $L^q(X;\V)$.
Dependence of constants on parameters relevant to limiting statements is displayed explicitly, as in $A_{T,\kappa}$; harmless dependence on the fixed geometric and operator data is suppressed.

\section{Abstract framework and kernel identities used in the proof}\label{sec:framework}

\subsection{Setting and hypotheses}\label{subsec:hyp}

We work on a locally compact, separable metric measure space $(X,d,\mu)$, where $\mu$ is a Radon measure with $\supp\mu=X$ and $0<\mu(B(x,r))<+\infty$ for every $x\in X$ and $r>0$. We assume the \emph{doubling volume property}: there exists $C_D>0$ such that
\begin{equation}\label{framework-VD}
\forall x\in X,\ \forall r>0,\quad \mu(B(x,2r))\leq C_D\,\mu(B(x,r)).
\end{equation}
This classical assumption is satisfied on $\R^d$, on compact Riemannian and sub-Riemannian manifolds, and on compact metric graphs (see \cite{AgrachevBarilariBoscain2019,Grigoryan2009_HeatKernel,sikora2004riesz}). We fix a Euclidean vector bundle $\V\to X$ of rank $r_\V=\rank(\V)$, endowed with the fiberwise inner product $\langle\cdot,\cdot\rangle_{\V_x}$ and norm $|\cdot|_{\V_x}$, and denote by $L^p(X;\V)=L^p(X,\mu;\V)$ the usual Banach spaces of $L^p$ sections.

We consider an unbounded, self-adjoint, nonnegative operator $A$ on $L^2(X;\V)$ satisfying:
\begin{enumerate}[label=(A\arabic*)]
\item\label{item:A1} [\emph{Ultracontractivity}] For every $t>0$,
$e^{-tA}\in\mathcal{L}(L^2(X;\V),L^\infty(X;\V))$, and there exist $\alpha_1>0$ and $C>0$ such that
\[
\forall t\in(0,1],\quad \|e^{-tA}\|_{2\to\infty}\leq C\,t^{-\alpha_1/2}.
\]
\item\label{item:A2} [\emph{Davies--Gaffney}] For all measurable $E,F\subset X$ and $t>0$,
\[
\|1_F e^{-tA}1_E\|_{\mathcal{L}(L^2(X;\V))}\leq e^{-d(E,F)^2/(4t)}.
\]
\item\label{item:A3} [\emph{Local Weyl law}] There exist $\alpha_2>0$, a slowly varying function $\chi_2 : [0,+\infty)\to(0,+\infty)$ (i.e., $\lim_{\lambda\to+\infty}\chi_2(a\lambda)/\chi_2(\lambda)=1$ for every $a>0$), and a measurable function $c:X\to(0,+\infty)$ such that, for $\mu$-a.e.\ $x\in X$,
\begin{equation}\label{loc-Weyl}
E_x(\lambda)\underset{\lambda\to+\infty}{\sim} c(x)\,\lambda^{\alpha_2}\chi_2(\lambda),
\end{equation}
where $E_x(\lambda):=\tr_{\V_x}\bigl(K_{1_{[0,\lambda]}(A)}(x,x)\bigr)$ and $K_{1_{[0,\lambda]}(A)}(x,y)\in\Hom(\V_y,\V_x)$ is the integral kernel of $1_{[0,\lambda]}(A)$.
\end{enumerate}

These three assumptions are verified in all the geometric settings considered in the applications given in Section~\ref{sec:applications}. Let us briefly comment on each.

The qualitative ultracontractivity property $e^{-tA}:L^2(X;\V)\to L^\infty(X;\V)$ for every fixed $t>0$ is classical in heat-kernel analysis; on Riemannian manifolds and on $\R^d$ it follows, for instance, from standard Sobolev estimates (see \cite{GrigoryanHu2014DV,Grigoryan2009_HeatKernel}). By Proposition~\ref{prop:heat-kernel}, this qualitative property is equivalent to the existence, for each $t>0$, of a bounded heat kernel $p_t(x,y)$. Assumption~\ref{item:A1} additionally records the quantitative polynomial rate as $t\downarrow0$. Up to the finite fiber rank, this rate is equivalent to the corresponding on-diagonal estimate: see \eqref{diag_trace_vs_L2} and Remark~\ref{rem:hk-conv}. We do not assume any continuity of $p_t$. This is essential for settings such as $\delta'$-coupled Laplacians on metric graphs, where eigenfunctions may be discontinuous at vertices and the heat kernel is generally not jointly continuous on $X\times X$; see Subsection~\ref{subsec:graphs-app}.

Assumption~\ref{item:A2} is the standard Davies--Gaffney off-diagonal estimate. By \cite{sikora2004riesz}, it is equivalent to finite propagation speed for the wave equation associated with $A$. Together with on-diagonal estimates, it is also the starting point of the abstract Gaussian upper-bound theory of \cite{CoulhonSikora2008}. We use its finite-propagation content directly through the exact wave-to-heat (or heat-from-wave) Kannai transmutation formula. This avoids assuming pointwise Gaussian estimates while retaining the geometric scale responsible for the threshold $1/2$. The assumption holds in our examples either by Lipschitz-weight energy arguments (Riemannian and graph cases) or by general results (sub-Riemannian sub-Laplacians \cite{Melrose1986PropagationWaveGroupSubelliptic}, and nonnegative Schr\"odinger operators on $\R^d$ via Feynman--Kac). With our normalization, the associated wave equation propagates at speed at most one.

Assumption~\ref{item:A3} is a form of local Weyl law on the diagonal. On a smooth compact Riemannian manifold it reduces to the classical H\"ormander spectral asymptotics \cite{Hormander1968Acta}, with $\alpha_2=d/2$ and $\chi_2\equiv1$; the coefficient $c$ is smooth and positive, and is constant in the scalar unweighted case. It is used only to obtain a lower bound on the $L^2(X;\V)$ norm of our test section. The slowly varying correction $\chi_2$ is included for flexibility: logarithmic factors may occur in global Weyl laws of non-equiregular sub-Riemannian structures \cite{CdVHT_AHL_2021,ColinDeVerdiereHillairetTrelat2022arXiv2212_02920}. In all the applications below, however, the pointwise law at $\mu$-a.e.\ point is purely polynomial ($\chi_2\equiv1$). In equiregular sub-Riemannian settings one has $\alpha_1=\alpha_2$, whereas the Grushin-type example has $\alpha_1>\alpha_2$. Concrete values of $(\alpha_2,\chi_2,c)$, together with $\alpha_1$ where it differs, are summarized in the table at the beginning of Section~\ref{sec:applications}.

\begin{remark}\label{rem:A3-remarks}
Two comments on~\ref{item:A3} are in order.
First, $c$ is not assumed to be bounded: in the almost-Riemannian example of Subsection~\ref{subsec:sR-app}, $c$ blows up near the singular set. Second, Assumption~\ref{item:A3} is used only at a single, well-chosen base point (see the proof of Theorem~\ref{thm:main_thm}). It would suffice to assume that, for every $\delta<\mathcal{L}(\omega)$, the set of points $x\in U_\delta$ at which \eqref{loc-Weyl} holds has positive $\mu$-measure. We keep the $\mu$-a.e.\ formulation for readability.
\end{remark}

We emphasize that we assume neither compact resolvent nor that $A$ derives from a \emph{regular strongly local} Dirichlet form (which are the standard assumptions in the literature, see \cite{Davies1989,DaviesSimon1984,GrigoryanHu2014DV,GrigorYan2014HKMMS}), so that our framework covers a wide variety of settings, including examples with non-continuous kernels, such as $\delta'$-coupled Laplacians on metric graphs (see Section~\ref{subsec:graphs-app}).
This is why we develop a kernel calculus tailored to vector bundles and to kernels defined only $\Delta$-a.e.; it is outlined in the next subsection and proved in Appendix~\ref{app:kernel-toolbox}.

\subsection{Overview of kernel calculus}\label{subsec:kernel-id}

This subsection presents, in a self-contained but minimal form, the kernel-calculus identities used in the proof of Theorem~\ref{thm:main_thm}.
On a smooth compact manifold, compact resolvent and smooth eigenfunctions make the same objects explicit through familiar eigenfunction expansions. We package the abstract framework so that the reader may keep this model picture in mind and proceed directly to Section~\ref{sec:main_proof} without entering the appendix.
We emphasize, however, that extending these objects to our generality is not just a matter of formalism: the constructions are genuinely subtler, and the careful development carried out in Appendix~\ref{app:kernel-toolbox}, to which we refer for complete statements and proofs, is what makes them rigorous in the absence of compact resolvent or kernel continuity.

The proof of Theorem~\ref{thm:main_thm} is based on the analysis of a single test section, the spectral packet $\widetilde g_{r,t,x}=K_{(\cosh(r\sqrt A)-1)e^{-tA}}(\cdot,x)$, and crucially uses two ingredients: a closed-form expression for its $L^2$-norm in terms of the spectral data of $A$ at the point $x$ (used as a lower bound via Assumption~\ref{item:A3}), and pointwise estimates on $\widetilde g_{r,t,x}(y)$ for $y$ away from $x$ via the wave kernel and the transmutation formula (used as an upper bound on $\omega$ via Assumption~\ref{item:A2}).
 In our generality, building these tools is delicate, as two classical mechanisms are unavailable:
\begin{itemize}
    \item We do \emph{not} assume that $A$ has compact resolvent, so there need not be a discrete eigenbasis $(\varphi_n)$ from which to form the pointwise expansion $\sum_n\phi(\lambda_n)\varphi_n(x)\otimes\varphi_n(y)^\flat$ (where $^\flat$ denotes the Riesz identification).
    \item We do \emph{not} assume that $p_t$ (or any other $K_{\phi(A)}$) is continuous, so pointwise diagonal values $K(x,x)$ are not defined in the usual sense.
\end{itemize}
A priori, the conjunction of these two facts even prevents one from giving a meaning to the diagonal restriction $K(x,x)$ for a.e.\ $x$. In the familiar smooth compact-resolvent setting, pointwise eigenfunction representatives provide such a diagonal, and continuity does so as well; without either mechanism, the ordinary $\mu\otimes\mu$-a.e.\ equivalence class of a kernel is too coarse.

\medskip\noindent\textbf{A canonical framework for diagonals.}
This first issue is in fact resolved relatively simply. The key observation, established in Proposition~\ref{prop:heat-kernel} under~\ref{item:A1} alone, is that the heat kernel $p_t$ — although not continuous in general — is canonically defined not merely $\mu\otimes\mu$-a.e.,
 but on a full-measure square $(X\setminus N)^2$ for some $\mu$-null set $N\subset X$, with the symmetry $p_t(x,y)=p_t(y,x)^*$ (where $^*$ denotes the adjoint of the linear map $p_t(y,x)\in\Hom(\V_x,\V_y)$, so that $p_t(y,x)^*\in\Hom(\V_y,\V_x)$) and the semigroup identity holding on this square. We refer to this stronger form of equivalence as \emph{$\Delta$-almost everywhere} ($\Delta$-a.e.\ for short, see Definition~\ref{def:Delta-ae}). Its key feature is that it makes diagonal restrictions canonical: if $K$ is bounded and defined $\Delta$-a.e., then the diagonal $x\mapsto K(x,x)$ is unambiguously defined for $\mu$-a.e.\ $x\in X$, independently of the chosen representative. By Proposition~\ref{prop:heat-kernel}, the heat kernel is naturally a $\Delta$-a.e.\ defined object, and as we shall see this canonical structure propagates to all the kernels we consider in the sequel. From now on, all kernels $K(x,y)\in\Hom(\V_y,\V_x)$ are tacitly understood as $\Delta$-a.e.\ defined.

\medskip\noindent\textbf{Pointwise spectral measures.} The second issue, the absence of compact resolvent, is the genuine obstacle, and one of the main technical devices of the paper is a substitute for the spectral expansion.
In a smooth compact-resolvent setting, the kernel $K_{\phi(A)}$ has the familiar eigenfunction expansion $\sum_n\phi(\lambda_n)\varphi_n(x)\otimes\varphi_n(y)^\flat$,
which can be rewritten in measure-theoretic form as $K_{\phi(A)}(x,y)=\int\phi(\lambda)\,d\Pi_{x,y}(\lambda)$, where $\Pi_{x,y}=\sum_n\delta_{\lambda_n}\,\varphi_n(x)\otimes\varphi_n(y)^\flat$ is a $\Hom(\V_y,\V_x)$-valued measure on $[0,+\infty)$ supported on the spectrum. We construct such an object in our general setting directly from the heat kernel and the functional calculus of $A$, without relying on an eigenfunction expansion or compact resolvent.
More precisely, Proposition~\ref{prop:spectral-mes} shows that, under~\ref{item:A1} alone, there exists a unique family of locally finite $\Hom(\V_y,\V_x)$-valued (signed) measures $\Pi_{x,y}$ on $[0,+\infty)$, defined for $\Delta$-a.e.\ $(x,y)\in X\times X$ by an explicit formula involving $p_{t/2}$ and the bounded operators $1_B(A)e^{tA}$ for bounded Borel $B\subset[0,+\infty)$. We call $\Pi_{x,y}$ the \emph{pointwise spectral measure} of $A$ at $(x,y)$. Note that the canonical $\Delta$-a.e.\ structure ensures, in particular, that the diagonal $\Pi_{x,x}$ is well defined for $\mu$-a.e.\ $x$.

The pointwise spectral measure provides a unified pointwise spectral representation for functions of $A$. By Proposition~\ref{prop:spectral-mes}, for every bounded compactly supported Borel $\phi$, the operator $\phi(A)$ has integral kernel
\begin{equation}\label{kernel-formula}
K_{\phi(A)}(x,y)=\int_0^{+\infty}\phi(\lambda)\,d\Pi_{x,y}(\lambda)\qquad\Delta\text{-a.e.},
\end{equation}
and defining $\nu_x=\tr_{\V_x}(\Pi_{x,x})$, we get the following identity for $\mu$-a.e.\ $x\in X$:
\begin{equation}\label{plancherel}
\int_X\HSnorm{K_{\phi(A)}(x,y)}^2\,d\mu(y)=\int_0^{+\infty}|\phi(\lambda)|^2\,d\nu_x(\lambda).
\end{equation}
This identity is precisely what allows us to express $L^2$-norms of kernel-type test functions in terms of intrinsic spectral data of $A$ at the point $x$. Furthermore, the scalar spectral function
\[E_x(\lambda)=\nu_x([0,\lambda])=\tr_{\V_x}\bigl(K_{1_{[0,\lambda]}(A)}(x,x)\bigr)\]
generalizes the classical H\"ormander spectral function to our setting, and Assumption~\ref{item:A3} reads exactly $E_x(\lambda)\sim c(x)\lambda^{\alpha_2}\chi_2(\lambda)$ as $\lambda\to+\infty$.
From the quantitative ultracontractivity~\ref{item:A1}, we prove Proposition~\ref{prop:kernel-weyl}, which extends the kernel formula~\eqref{kernel-formula} to every measurable $\phi$ with $|\phi(\lambda)|\leq C(1+\lambda)^{-\alpha_1-\varepsilon}$. In particular, this applies to $\phi(\lambda)=\big(\cosh(r\sqrt\lambda)-1\big)e^{-t\lambda}$, which decays faster than any power of $\lambda$, and thus defines the test section $\widetilde g_{r,t,x}$.

Now applying~\eqref{plancherel} to the spectral packet $\phi(\lambda)=\big(\cosh(r\sqrt\lambda)-1\big)e^{-t\lambda}$, we obtain one of the key identities at the heart of the proof, which we call the \emph{pointwise Plancherel identity}: for $\mu$-a.e.\ $x\in X$ and every $r,t>0$,
\begin{equation*}
\int_X\HSnorm{\widetilde g_{r,t,x}(y)}^2\,d\mu(y)=\int_0^{+\infty}\big(\cosh(r\sqrt\lambda)-1\big)^2 e^{-2t\lambda}\,d\nu_x(\lambda).
\end{equation*}
It plays the role that Plancherel's theorem plays in the Euclidean Fourier computation. This is the content of Lemma~\ref{resummed-transmutation}; combined with the local Weyl asymptotics of $\nu_x$ (Assumption~\ref{item:A3}), it provides the lower bound on the $L^2(X)$-norm of $\widetilde g_{r,t,x}$ established in Lemma~\ref{resummed-lower}.

\medskip\noindent\textbf{Wave kernel and transmutation.} The complementary upper bound on $\HSnorm{\widetilde g_{r,t,x}(y)}$ for $y$ away from $x$ (Lemma~\ref{resummed-omega}) requires off-diagonal information, provided by Assumption~\ref{item:A2}. Using again the kernel formula~\eqref{kernel-formula} with the bounded oscillatory function $C_\psi(\lambda)=\int_\R\psi(s)\cos(s\sqrt\lambda)\,ds$ for $\psi$ smooth with sufficient decay, one constructs in Appendix~\ref{app:kernel-toolbox} a weak wave kernel $w(x,y)$, which is a $\Hom(\V_y,\V_x)$-valued distribution on $\R$ defined for $\Delta$-a.e.\ $(x,y)$ (Proposition~\ref{prop:weak-wave}). It is related to the heat kernel by the Kannai transmutation formula (Lemma~\ref{lem:transmutation}):
\[p_t(x,y)=\langle w(x,y),\gamma_t\rangle,\qquad\gamma_t(s)=(4\pi t)^{-1/2}e^{-s^2/(4t)}.\]
Moreover, Proposition~\ref{prop:weak-wave} shows that $w$ satisfies the wave equation $A\langle w,\psi\rangle=-\langle w,\psi''\rangle$ ($\Delta$-a.e.) in the weak sense for every test function $\psi$.
Thanks to the transmutation formula (Lemma~\ref{resummed-transmutation}), the whole packet is represented through the wave kernel against a Schwartz weight: for $\Delta$-a.e.\ $(x,y)$,
\begin{equation}\label{transmutation-Ak}
\widetilde g_{r,t,x}(y)=\big\langle w(y,x),\,\Phi_{r,t}-\gamma_t\big\rangle,\qquad \Phi_{r,t}(s)=(4\pi t)^{-1/2}e^{(r^2-s^2)/(4t)}\cos\!\big(\tfrac{rs}{2t}\big).
\end{equation}
The role of Assumption~\ref{item:A2} is then captured by the support property
\[
\supp w(x,y)\subset\R\setminus(-d(x,y),d(x,y))\qquad \text{for $\mu\otimes\mu$-a.e.\ $(x,y)$},
\]
deduced from Proposition~\ref{prop:fsp_wave_ker} in Remark~\ref{rem:finite-speed}. Since the wave kernel issued from $x$ reaches a point $y$ only for $|s|\geq d(x,y)$, while $\Phi_{r,t}-\gamma_t$ is exponentially small for $|s|\geq d(x,y)>r$, the pairing~\eqref{transmutation-Ak} is exponentially small on $\omega$; this is the mechanism behind Lemma~\ref{resummed-omega}.
Intuitively, a singularity placed at distance $d(x,\omega)$ from the observation set cannot be detected by the wave flow before that time; after transmutation, this geometric delay becomes the exponential barrier for the heat flow.

\section{Proof of the geometric barrier}\label{sec:main_proof}
In this section, we prove Theorem~\ref{thm:main_thm} and its corollaries: we first describe the strategy, then establish the key lemmas, and finally assemble the proof.

Each assumption plays a distinct role. Assumption~\ref{item:A1} provides the heat kernel and the pointwise spectral measure, hence the pointwise Plancherel identity for the columns of $K_{\phi(A)}$, and it ensures the existence of the relevant kernels and their membership in $L^\infty_\Delta$ (used throughout the toolbox of Appendix~\ref{app:kernel-toolbox}). Assumption~\ref{item:A2}, through the transmutation formula and the finite speed of propagation of the wave kernel, makes the test object exponentially small on $\omega$. Assumption~\ref{item:A3}, the pointwise Weyl law, is used only at a single base point, to bound the test object from below.

\paragraph{Proof outline.}
The proof is driven by the single test section $\widetilde g_{r,t,x}=K_{(\cosh(r\sqrt A)-1)\,e^{-tA}}(\cdot,x)$, for a base point $x$ and a radius $0<r<d_{\mathrm{ess}}(x,\omega)$, which is a sharp wave packet of Paley--Wiener type: in the $\sqrt A$ variable, $\cosh(r\sqrt\lambda)$ is the even multiplier of exponential type $r$ with maximal growth along the spectrum, so $\widetilde g_{r,t,x}$ amplifies the heat semigroup as much as a propagation radius $r$ allows, and no more. The strategy is to show that this packet is \emph{sharp} in two opposite directions and to read off the resulting scale. Lemma~\ref{resummed-transmutation} provides the spectral identity defining $\widetilde g_{r,t,x}$ and the pointwise Plancherel formula for its norm; the two competing estimates are then:
\begin{itemize}
\item a \emph{lower bound}, expressing that the packet is genuinely large in $L^2(X)$: by Lemma~\ref{resummed-lower}, $\|\widetilde g_{r,t,x}\|_{L^2(X)}^2\gtrsim e^{(1-\eta)\,r^2/(2t)}$ as $t\to0$, the rate $r^2/(2t)$ being fixed by the pointwise Weyl density at $x$ through~\ref{item:A3};
\item an \emph{upper bound on $\omega$}, expressing that the packet is exponentially small there: since its propagation radius $r$ is strictly below $\delta:=d_{\mathrm{ess}}(x,\omega)$, finite speed of propagation (Assumption~\ref{item:A2}) gives $\int_\omega\HSnorm{\widetilde g_{r,t,x}}^2\lesssim e^{-(\delta^2-r^2)/(2t)}$ for small $t$ (Lemma~\ref{resummed-omega}), the large-time regime being handled by Lemma~\ref{resummed-large-t}.
\end{itemize}
We apply \eqref{gen-obs} to the regularized packet $f_{s,\ell}=\widetilde g_{r,s,x}e_\ell$, choose $s=\tau^2$, and restrict the resulting time integral to $u\in[\tau+s,(1+\varepsilon)\tau+s]$. The two estimates above then give directly $h(\tau)\leq A e^{-\kappa r^2/\tau}$ for every $0<\kappa<1/2$. The matching exponential scales in the lower and upper bounds allow $r\uparrow d_{\mathrm{ess}}(x,\omega)$; optimizing the base point then yields the barrier $\mathcal{L}(\omega)^2/t$ with the universal threshold $1/2$. Equivalently, the resulting family of quantitative estimates gives the logarithmic endpoint \eqref{main-varadhan}.

\subsection{Key lemmas}\label{subsec:key-lemmas}

We start by establishing the spectral identity underlying the test section $\widetilde g_{r,t,x}$. The pointwise Plancherel formula is then recalled immediately after the lemma.

\begin{lem}[Packet transmutation]\label{resummed-transmutation}
For all $r > 0$ and $t > 0$, the function
\begin{equation}\label{Phi-def}
\Phi_{r,t}(s) = \frac{1}{\sqrt{4\pi t}}\, e^{(r^2 - s^2)/(4t)}\cos\Big(\frac{r s}{2 t}\Big), \qquad s \in \R,
\end{equation}
belongs to $\mathcal{S}(\R)$, and one has, for every $\lambda \geq 0$,
\begin{equation}\label{cosh-Cpsi}
\cosh(r\sqrt{\lambda})\, e^{-t\lambda} = \mathcal{C}_{\Phi_{r,t}}(\lambda), \qquad \big(\cosh(r\sqrt{\lambda}) - 1\big)\, e^{-t\lambda} = \mathcal{C}_{\Phi_{r,t} - \gamma_t}(\lambda).
\end{equation}
\end{lem}

\begin{proof}
Using \eqref{Phi-def}, we have $\Phi_{r,t} \in \mathcal{S}(\R)$. Recall that $\int_\R e^{-s^2/(4t)}\cos(s\omega)\, ds = \sqrt{4\pi t}\, e^{-t\omega^2}$ for every $\omega \in \R$. Using $\cos(a)\cos(b) = \frac{1}{2}\big(\cos(a - b) + \cos(a + b)\big)$ with $a = r s/(2t)$ and $b = s\sqrt{\lambda}$, we obtain
\begin{align*}
\mathcal{C}_{\Phi_{r,t}}(\lambda)
&= \frac{e^{r^2/(4t)}}{\sqrt{4\pi t}} \int_\R e^{-s^2/(4t)}\cos\Big(\frac{r s}{2 t}\Big)\cos(s\sqrt{\lambda})\, ds \\
&= \frac{e^{r^2/(4t)}}{2\sqrt{4\pi t}} \int_\R e^{-s^2/(4t)}\Big(\cos\big(s(\frac{r}{2t} - \sqrt{\lambda})\big) + \cos\big(s(\frac{r}{2t} + \sqrt{\lambda})\big)\Big)\, ds \\
&= \frac{e^{r^2/(4t)}}{2}\Big(e^{-t(\frac{r}{2t} - \sqrt{\lambda})^2} + e^{-t(\frac{r}{2t} + \sqrt{\lambda})^2}\Big) = \frac{1}{2}\big(e^{r\sqrt{\lambda}} + e^{-r\sqrt{\lambda}}\big)\, e^{-t\lambda} = \cosh(r\sqrt{\lambda})\, e^{-t\lambda},
\end{align*}
where we used $-t\big(\frac{r}{2t} \mp \sqrt{\lambda}\big)^2 = -\frac{r^2}{4t} \pm r\sqrt{\lambda} - t\lambda$. The second identity in \eqref{cosh-Cpsi} follows by subtracting $\mathcal{C}_{\gamma_t}(\lambda) = e^{-t\lambda}$ (Lemma~\ref{lem:transmutation}).
\end{proof}

For $r > 0$ and $t > 0$, the function $\lambda \mapsto \big(\cosh(r\sqrt{\lambda}) - 1\big)e^{-t\lambda}$ decays faster than any negative power of $\lambda$; hence, by Proposition~\ref{prop:kernel-weyl}, it defines a kernel in $L^\infty_\Delta$, and for a.e.\ $x \in X$ we have
\[
\widetilde{g}_{r,t,x}(\cdot) = K_{(\cosh(r\sqrt{A}) - 1)e^{-tA}}(\cdot,x) = \sum_{k \geq 1} \frac{r^{2k}}{(2k)!}\, (A^k p_t)(\cdot,x) \in L^2(X;\V).
\]
Here $A^kp_t(\cdot,x)=K_{\lambda^ke^{-t\lambda}(A)}(\cdot,x)$ in the sense of Corollary~\ref{cor:A-on-kernel}, and the series converges in $L^2(X;\V)$: the partial sums correspond to the multipliers $\phi_n(\lambda)=\sum_{k=1}^n \frac{(r\sqrt\lambda)^{2k}}{(2k)!}\,e^{-t\lambda}$, which converge pointwise to $(\cosh(r\sqrt\lambda)-1)e^{-t\lambda}$ and are dominated by $\cosh(r\sqrt\lambda)\,e^{-t\lambda}\in L^2(\nu_x)$, so that the pointwise Plancherel identity \eqref{L2rowHS} yields convergence of the columns in $L^2(X;\V)$.
By the pointwise Plancherel identity \eqref{L2rowHS}, for a.e. $x \in X$,
\begin{equation}\label{gtilde-norm}
\int_X \HSnorm{\widetilde{g}_{r,t,x}(y)}^2\, d\mu(y) = \int_0^{+\infty} \big(\cosh(r\sqrt{\lambda}) - 1\big)^2 e^{-2t\lambda}\, d\nu_x(\lambda),
\end{equation}
while, by Lemma~\ref{resummed-transmutation} and the definition of $w$, for $\Delta$-a.e. $(x,y)$,
\begin{equation}\label{gtilde-wave}
\widetilde{g}_{r,t,x}(y) = \scalp{w(y,x)}{\Phi_{r,t} - \gamma_t}.
\end{equation}
We have the following two-sided control of $\widetilde{g}_{r,t,x}$.

\begin{lem}[Lower bound, small time]\label{resummed-lower}
For a.e. $x \in X$, every $\eta \in (0,1)$ and every $r > 0$, there exist $c_\eta>0$ (independent of $r$ and $x$) and $\tau_{\eta,r,x}>0$ such that, for every $t\in(0,\tau_{\eta,r,x})$,
\begin{equation}\label{lowerA}
\int_X \HSnorm{\widetilde{g}_{r,t,x}(y)}^2\, d\mu(y) \geq c_\eta\, c(x)\, e^{(1 - \eta)\, r^2/(2t)}.
\end{equation}
\end{lem}

\begin{proof}
Fix $\eta \in (0,1)$. Since $\cosh(z) - 1 \geq \frac{1}{4} e^z$ for $z\geq z_0:=\ln4$, we have $\big(\cosh(r\sqrt{\lambda}) - 1\big)^2 \geq \frac{1}{16} e^{2 r\sqrt{\lambda}}$ for $\lambda \geq \lambda_0 = (z_0/r)^2$. The function $\lambda\mapsto\varphi_t(\lambda) = 2 r\sqrt{\lambda} - 2 t\lambda$ attains its maximum $r^2/(2t)$ at $\lambda_\star = r^2/(4 t^2)$, and $\varphi_t''(\lambda) = -\frac{1}{2}\, r\, \lambda^{-3/2}$. By \eqref{gtilde-norm},
\begin{equation*}
\int_X \HSnorm{\widetilde{g}_{r,t,x}(y)}^2\, d\mu(y) \geq \frac{1}{16}\int_{\lambda \geq \lambda_0} e^{\varphi_t(\lambda)}\, d\nu_x(\lambda).
\end{equation*}
By the local Weyl law \eqref{loc-Weyl}, one may choose $s \in (0,1)$ (depending only on $\eta$, to be adjusted below) and, for a.e.\ $x$, a threshold $\Lambda_{\eta,x} > 0$ such that, for all $\Lambda \geq \Lambda_{\eta,x}$,
\begin{equation}\label{weyl-window}
E_x\big((1 + s)\Lambda\big) - E_x\big((1 - s)\Lambda\big) \geq K_\eta\, c(x)\, \Lambda^{\alpha_2} \chi_2(\Lambda).
\end{equation}
Indeed, by \eqref{loc-Weyl} and the slow variation of $\chi_2$, the left-hand side is equivalent, as $\Lambda\to+\infty$, to $c(x)\big((1+s)^{\alpha_2}-(1-s)^{\alpha_2}\big)\Lambda^{\alpha_2}\chi_2(\Lambda)$, so that one may take $K_\eta=\tfrac{1}{2}\big((1+s)^{\alpha_2}-(1-s)^{\alpha_2}\big)$, which depends only on $\eta$ through $s$.
Take $\tau_{\eta,r,x}>0$ small enough that $t<\tau_{\eta,r,x}$ implies $\lambda_\star \geq \max(\lambda_0/(1-s), \Lambda_{\eta,x})$, and restrict the integral to $I = [(1 - s)\lambda_\star,\, (1 + s)\lambda_\star]$. A second-order expansion of $\varphi_t$ at $\lambda_\star$ gives $\varphi_t(\lambda) \geq (1 - C s^2)\, r^2/(2t)$ on $I$, for some constant $C$. Shrinking $s$ if necessary (depending only on $\eta$), this reads $\varphi_t(\lambda) \geq (1 - \eta/2)\, r^2/(2t)$ on $I$. Combining with \eqref{weyl-window} at $\Lambda = \lambda_\star$,
\begin{equation*}
\int_X \HSnorm{\widetilde{g}_{r,t,x}(y)}^2\, d\mu(y) \geq \frac{K_\eta}{16}\, c(x)\, \lambda_\star^{\alpha_2} \chi_2(\lambda_\star)\, e^{(1 - \eta/2)\, r^2/(2t)}.
\end{equation*}
For fixed $r$ and $t \to 0$, the prefactor $\lambda_\star^{\alpha_2} \chi_2(\lambda_\star) = \big(r^2/(4 t^2)\big)^{\alpha_2} \chi_2\big(r^2/(4 t^2)\big)$ is subexponential in $1/t$, hence bounded below by $e^{-(\eta/2)\, r^2/(2t)}$ after shrinking $\tau_{\eta,r,x}$; this yields \eqref{lowerA} with $c_\eta = K_\eta/16$.
\end{proof}

\begin{lem}[Upper bound on $\omega$, small time]\label{resummed-omega}
Let $m \in \N^*$ with $m>\alpha_1$. For a.e. $x \in X$ with $\delta := d_{\mathrm{ess}}(x,\omega) > 0$, and for every $r \in (0,\delta)$, there exists $C_x > 0$ (depending also on $r$ and $m$) such that, for all $t \in (0, \frac{\delta^2}{4(m+1)}  )$,
\begin{equation}\label{upperB}
\int_\omega \HSnorm{\widetilde{g}_{r,t,x}(y)}^2\, d\mu(y) \leq C_x\, \big(1 + t^{-2(2m + 1)}\big)\, e^{-(\delta^2 - r^2)/(2t)}.
\end{equation}
\end{lem}
\begin{proof}
Set $\psi_t = \Phi_{r,t} - \gamma_t \in \mathcal{S}(\R)$. By \eqref{Cpsi_bound_by_pk}, for every $\psi \in \W$ one has $\HSnorm{\scalp{w(x,y)}{\psi}} \leq C_m\big(\Vert \psi \Vert_{L^1(\R)} + \Vert \psi^{(2m)} \Vert_{L^1(\R)}\big)$. Since $\langle\,\cdot\,\rangle^{-2}\in L^1(\R)$, each $L^1$ norm is controlled by the corresponding weighted $L^\infty$ norm; hence, after modifying $C_m$,
\begin{equation*}
\HSnorm{\scalp{w(x,y)}{\psi}} \leq C_m\Big(\Vert \langle\,\cdot\,\rangle^2 \psi \Vert_{L^\infty(\R)} + \Vert \langle\,\cdot\,\rangle^2 \psi^{(2m)} \Vert_{L^\infty(\R)}\Big).
\end{equation*}
By Remark~\ref{rem:finite-speed}, $w(x,y)$ is supported in $\{|s|\geq d(x,y)\}$ for $\mu\otimes\mu$-a.e.\ $(x,y)$. Lemma~\ref{lem:localize_weighted_sup_norms} therefore localizes this estimate, for such pairs and with $\rho=d(x,y)$, to
\begin{equation}\label{localizedB}
\HSnorm{\scalp{w(y,x)}{\psi}} \leq C_m\, (1 + \rho^{-2m})\Big( \big\Vert \langle\,\cdot\,\rangle^2 \psi \big\Vert_{L^\infty(\vert s \vert \geq \rho)} + \big\Vert \langle\,\cdot\,\rangle^2 \psi^{(2m)} \big\Vert_{L^\infty(\vert s \vert \geq \rho)} \Big).
\end{equation}
Let us estimate the right-hand side for $\psi = \psi_t$ and $\rho \geq \delta$. Both $\Phi_{r,t}$ and $\gamma_t$ are of the form $(4\pi t)^{-1/2} e^{(a - s^2)/(4t)}\cos(b s)$ with $(a,b) = (r^2,\, r/(2t))$ or $(0,0)$. Hence, for every $N \in \N$, there is a polynomial $Q_N$ with nonnegative coefficients (after bounding by their absolute values the coefficients produced by differentiation) such that
\begin{equation*}
\big\vert \partial_s^N \psi_t(s) \big\vert \leq \frac{1}{\sqrt{4\pi t}}\, e^{(r^2 - s^2)/(4t)}\, Q_N\big(\vert s \vert,\, 1/t\big),
\end{equation*}
the term carrying $e^{r^2/(4t)}$ being the larger one.
For $s\neq 0$, set $P_N(s):=\langle s\rangle^{2}Q_N(|s|,1/t)$ and $g_N(s):=P_N(s)\,e^{-s^{2}/(4t)}$ for $s\ge 0$, so that $\langle s\rangle^{2}\lvert\partial_s^{N}\psi_t(s)\rvert\le (4\pi t)^{-1/2}e^{r^{2}/(4t)}\,g_N(|s|)$. Since $Q_N$ has nonnegative coefficients, so does $P_N=\sum_{k=0}^{D}a_k s^{k}$, with $D:=\deg_s P_N\le N+2$; hence, for $s>0$, $P_N'(s)/P_N(s)\le D/s$, and $(\log g_N)'(s)\le D/s-s/(2t)<0$ whenever $s^{2}>2tD$. In particular, if $t<\delta^2/(4(m+1))\leq\delta^{2}/(2D)$ for $N\in\{0,2m\}$, then $g_N$ is nonincreasing on $[\delta,+\infty)$, and since $\rho\ge\delta$,
\[
\bigl\lVert\langle\cdot\rangle^{2}\partial_s^{N}\psi_t\bigr\rVert_{L^{\infty}(|s|\ge\rho)}
\le (4\pi t)^{-1/2}\,e^{r^{2}/(4t)}\,g_N(\rho).
\]
Moreover, $Q_N(\rho,1/t)\leq C_N(1+\rho)^{N}(1+t^{-N})$, so that $g_N(\rho)\leq C_N(1+\rho)^{N+2}(1+t^{-N})\,e^{-\rho^2/(4t)}$, and therefore
\[
  \bigl\lVert\langle\cdot\rangle^{2}\partial_s^{N}\psi_t\bigr\rVert_{L^{\infty}(|s|\ge\rho)}
  \le C_N\bigl(1+t^{-N-1}\bigr)\,(1+\rho)^{N+2}\,e^{(r^{2}-\rho^{2})/(4t)},
  \qquad N\in\{0,2m\}.
\]
Plugging this into \eqref{localizedB} with $\rho = d(x,y) \geq \delta$ for $\mu$-a.e.\ $y \in \omega$ (by definition of $\delta = d_{\mathrm{ess}}(x,\omega)$), using $1+\rho^{-2m}\leq 1+\delta^{-2m}$ and \eqref{gtilde-wave}, we get, for $\mu$-a.e.\ $y\in\omega$,
\begin{equation*}
\HSnorm{\widetilde{g}_{r,t,x}(y)} \leq C_{m,\delta}\, \big(1 + t^{-2m - 1}\big)\,\big(1+d(x,y)\big)^{2m+2}\, e^{(r^2 - d(x,y)^2)/(4t)}.
\end{equation*}
We now square and integrate over $\omega$. Writing $e^{(r^2-\rho^2)/(2t)}=e^{-(\delta^2-r^2)/(2t)}\,e^{-(\rho^2-\delta^2)/(2t)}$ with $\rho=d(x,y)$, and using that, for $\rho\geq\delta$ and $0<t\leq t_1:=\delta^2/(4(m+1))$, one has $e^{-(\rho^2-\delta^2)/(2t)}\leq e^{-(\rho^2-\delta^2)/(2t_1)}$, we obtain
\begin{multline*}
\int_\omega \HSnorm{\widetilde{g}_{r,t,x}(y)}^2\,d\mu(y) \\
\leq 2C_{m,\delta}^2\,\big(1+t^{-2(2m+1)}\big)\,e^{-(\delta^2-r^2)/(2t)}\;e^{\delta^2/(2t_1)}\!\int_X \big(1+d(x,y)\big)^{4m+4} e^{-d(x,y)^2/(2t_1)}\,d\mu(y),
\end{multline*}
and the last integral is finite by Lemma~\ref{lem:gauss-doubling-int}. This is \eqref{upperB}.
\end{proof}

\begin{lem}[Upper bound on $\omega$, large time]\label{resummed-large-t}
For a.e. $x \in X$ and every $r > 0$, there exists $C_r>0$, independent of $x$, such that, for all $t \geq \max(1, r)$,
\begin{equation}\label{largeB}
\int_\omega \HSnorm{\widetilde{g}_{r,t,x}(y)}^2\, d\mu(y) \leq \frac{C_r}{t^2}.
\end{equation}
\end{lem}

\begin{proof}
Restricting the spatial integral to $\omega$ only decreases the left-hand side of \eqref{gtilde-norm}. Set
\begin{equation*}
I_t(x):=\int_0^{+\infty}\big(\cosh(r\sqrt\lambda)-1\big)^2e^{-2t\lambda}\,d\nu_x(\lambda).
\end{equation*}
For $z\geq0$, one has $\cosh z-1\leq \frac{1}{2}z^2\cosh z$, hence $(\cosh(r\sqrt\lambda)-1)^2
\leq \frac{r^4\lambda^2}{4}e^{2r\sqrt\lambda}$.
If $t\geq r$, then $2r\sqrt\lambda-t\lambda\leq r^2/t\leq r$. Therefore
\begin{equation*}
I_t(x)\leq \frac{r^4e^r}{4} \int_0^{+\infty}\lambda^2e^{-t\lambda}\,d\nu_x(\lambda).
\end{equation*}
Stieltjes integration by parts, followed by the change of variables $u=t\lambda$, gives
\begin{multline*}
\int_0^{+\infty}\lambda^2e^{-t\lambda}\,d\nu_x(\lambda)
=-\int_0^{+\infty}E_x(\lambda)(2\lambda-t\lambda^2)e^{-t\lambda}\,d\lambda\\
\leq \frac1{t^2}\int_0^{+\infty}E_x(u/t)(2u+u^2)e^{-u}\,du
\leq \frac1{t^2}\int_0^{+\infty}E_x(u)(2u+u^2)e^{-u}\,du
\leq \frac{C}{t^2}.
\end{multline*}
Here we used $t\geq1$, the monotonicity of $E_x$, and the uniform polynomial estimate \eqref{unif-bound-E}; in particular, $C$ is independent of $x$. This proves \eqref{largeB}, with a constant depending only on $r$ and the constants in~\ref{item:A1}.
\end{proof}

\subsection{Proof of Theorem~\ref{thm:main_thm}}\label{subsec:proof_main}

Fix $\eta \in (0,1)$. For a.e. $x \in X$ with $\delta := d_{\mathrm{ess}}(x,\omega) > 0$, fix $r \in (0,\delta)$.
For $s\in(0,1]$ and an orthonormal basis $(e_\ell)_{\ell=1}^{r_\V}$ of $\V_x$, consider the sections
\[
f_{s,\ell} := \widetilde g_{r,s,x}\,e_\ell = K_{(\cosh(r\sqrt A)-1)e^{-sA}}(\cdot,x)\,e_\ell.
\]

They belong to $L^2(X;\V)$ by Proposition~\ref{prop:kernel-weyl} and to $(\ker A)^\perp$, since their multiplier vanishes at $\lambda=0$. Moreover, Corollary~\ref{cor:semigroup-column} gives $e^{-tA}f_{s,\ell}=\widetilde g_{r,t+s,x}e_\ell$ for every $t>0$. For each fixed $s$, this identity holds for a.e.\ $x$. By first imposing it for $s\in\mathbb Q\cap(0,1]$ and then using $L^2$-continuity in $s$, we may choose one base point $x$ for which it holds simultaneously for all $s\in(0,1]$; no uncountable union of exceptional sets is involved.

For $s\in(0,1]$, set
\[
R_s:=\int_s^{T+s}\int_\omega\HSnorm{\widetilde g_{r,u,x}(y)}^2\,d\mu(y)\,du.
\]
We first prove that $M:=\sup_{s\in(0,1]}R_s<+\infty$. Let $m$ be as in Lemma~\ref{resummed-omega} and set
$a:=\frac{1}{2}\min\big(1,\tfrac{\delta^2}{4(m+1)}\big)$ and $b:=\max(1,r)$.
Then $0<a<b$. On $(0,a]$, Lemma~\ref{resummed-omega} supplies an integrable majorant at the origin. On $[a,b]$, the observed norm is bounded by the total norm $\int_X\HSnorm{\widetilde g_{r,u,x}}^2d\mu$, which is finite and nonincreasing in $u$ by \eqref{gtilde-norm}, hence bounded by its value at $a$. On $[b,+\infty)$, Lemma~\ref{resummed-large-t} supplies the integrable majorant $C_ru^{-2}$. Thus the integrand defining $R_s$ is bounded by a fixed integrable function on $(0,+\infty)$, uniformly in $s\in(0,1]$; for $T<+\infty$, only integrability on $(0,T+1]$ is needed.

In particular, the observed integral is finite, including when $T=+\infty$, so Definition~\ref{defn:h-obs} permits us to apply \eqref{gen-obs} to every $f_{s,\ell}$. Summing over $\ell$, which reconstructs the Hilbert--Schmidt norms, and changing variables $u=t+s$, we obtain
\begin{equation}\label{obs-resummed}
\int_s^{T+s} h(u-s) \int_X \HSnorm{\widetilde{g}_{r,u,x}(y)}^2\, d\mu(y)\, du \leq R_s\leq M.
\end{equation}

We now derive the pointwise estimate on $h$ directly, without taking a limit as $s\to0$. Fix $\varepsilon>0$, set $s=\tau^2$, and restrict the left-hand side of \eqref{obs-resummed} to $I_\tau=[\tau+s,(1+\varepsilon)\tau+s]$.
Choose $\tau>0$ so small that $(1+\varepsilon)\tau+\tau^2<\tau_{\eta,r,x}$, $(1+\varepsilon)\tau<T$ and $(1+\varepsilon)\tau\leq\tau_h$, with the middle condition omitted when $T=+\infty$. Then $I_\tau\subset[s,T+s]$, Lemma~\ref{resummed-lower} applies throughout $I_\tau$, and $h(u-s)\geq h(\tau)$ there. Hence \eqref{obs-resummed} gives
\begin{equation*}
M\geq\int_{I_\tau}h(u-s)\int_X\HSnorm{\widetilde g_{r,u,x}(y)}^2d\mu(y)du
\geq \varepsilon\tau\,c_\eta c(x)h(\tau)
\exp\Big(\frac{(1-\eta)r^2}{2((1+\varepsilon)\tau+\tau^2)}\Big).
\end{equation*}
Therefore
\[
h(\tau)\leq \frac{M}{\varepsilon c_\eta c(x)}\,\tau^{-1}
\exp\Big(-\frac{(1-\eta)r^2}{2(1+\varepsilon+\tau)\tau}\Big).
\]
Let $0<\kappa_0<(1-\eta)/(2(1+\varepsilon))$. For all sufficiently small $\tau$, the coefficient $(1-\eta)/(2(1+\varepsilon+\tau))$ exceeds $\kappa_0$ by a fixed positive amount; the factor $\tau^{-1}$ is therefore absorbed by the resulting exponential gap. Thus there exists $A_{\eta,\varepsilon,\kappa_0,r,x,T}>0$ such that
\[
h(\tau)\leq A_{\eta,\varepsilon,\kappa_0,r,x,T}e^{-\kappa_0r^2/\tau}
\]
for all sufficiently small $\tau>0$. Since $\eta$ and $\varepsilon$ are arbitrary, this conclusion holds for every $0<\kappa_0<1/2$.

It remains to optimize the geometry. Let $L:=\mathcal{L}(\omega)$. If $L=0$, the conclusion follows from the boundedness of $h$. Assume $0<L<+\infty$ and fix $0<\kappa<1/2$. Choose successively $\delta_0<L$, $0<\kappa_0<1/2$ and $r<\delta_0$, sufficiently close to $L$, $1/2$ and $\delta_0$, respectively, that $\kappa_0r^2>\kappa L^2$. The set $U_{\delta_0}=\{x\in X\mid d_{\mathrm{ess}}(x,\omega)>\delta_0\}$ is open and nonempty, hence has positive measure because $\supp\mu=X$. We may therefore choose $x\in U_{\delta_0}$ outside all exceptional null sets. The preceding estimate yields $h(t)\leq Ae^{-\kappa_0r^2/t}\leq Ae^{-\kappa L^2/t}$ for small $t$. If $L=+\infty$, the same argument, with any prescribed $R>0$ and with $\delta_0$ chosen sufficiently large, gives $h(t)\leq A_{T,R,\kappa}e^{-\kappa R^2/t}$ for small $t$. In both cases, the boundedness of $h$ extends the estimate to every $t\in(0,T)$ after enlarging the constant. This treats finite and infinite horizons simultaneously.
\qed

\subsection{Proof of Corollary~\ref{cor:conj_ervzua}}

Apply the $(\ker A)^\perp$ variant of Theorem~\ref{thm:main_thm} with $h(t)=C_\infty^{-1}e^{-\tilde\gamma_\infty/t}$ and $T=+\infty$. For every $0<\kappa<1/2$ and small $t$,
$C_\infty^{-1}e^{-\tilde\gamma_\infty/t}\leq A_{\infty,\kappa}e^{-\kappa\mathcal{L}(\omega)^2/t}$, hence $\tilde\gamma_\infty\geq\kappa\mathcal{L}(\omega)^2$. Letting $\kappa\uparrow1/2$ concludes. If $\mathcal{L}(\omega)=+\infty$, the finite-radius version of Theorem~\ref{thm:main_thm} yields in the same way $\tilde\gamma_\infty\geq\kappa R^2$ for every $R>0$, which is impossible for a finite $\tilde\gamma_\infty$. \qed

\subsection{Proof of Corollary~\ref{cor:fast-controls}}

Let $\widetilde K>K$. Multiplying \eqref{fast-obs} by $e^{-\widetilde K/T}$, using $\int_0^T\leq\int_0^{T_0}$ on the right-hand side, and integrating in $T\in(0,T_0]$, we obtain, for every $f\in L^2(X;\V)$,
\begin{multline*}
\int_0^{T_0} e^{-\widetilde K/T}\,\|e^{-TA}f\|_{L^2(X;\V)}^2\,dT
\leq C\Big(\int_0^{T_0}e^{-(\widetilde K-K)/T}\,dT\Big)\int_0^{T_0}\|e^{-tA}f\|_{L^2(\omega;\V)}^2\,dt \\
\leq C\,T_0\int_0^{T_0}\|e^{-tA}f\|_{L^2(\omega;\V)}^2\,dt.
\end{multline*}
Thus \eqref{gen-obs} holds with horizon $T_0$ and weight $h(t)=(CT_0)^{-1}e^{-\widetilde K/t}$. Theorem~\ref{thm:main_thm} gives, for every $0<\kappa<1/2$, $e^{-\widetilde K/t}\leq CT_0A_{T_0,\kappa}e^{-\kappa\mathcal{L}(\omega)^2/t}$ for small $t$. Hence $\widetilde K\geq\kappa\mathcal{L}(\omega)^2$; first letting $\kappa\uparrow1/2$ and then $\widetilde K\downarrow K$ proves the claim. \qed

\subsection{Proof of Corollary~\ref{cor:obs_sharp_cor}}

We prove the case $\mathcal{L}(\omega)<+\infty$; the infinite-distance statement follows from the finite-radius version of Theorem~\ref{thm:main_thm}. If $\mathcal{L}(\omega)=0$, the conclusion follows from the boundedness of $H_T$, so assume $\mathcal{L}(\omega)>0$. Let $0<\kappa<1/2$. Then
\[
H_T(\lambda)^2\leq A_{T,\kappa}\int_0^T e^{-\kappa\mathcal{L}(\omega)^2/t-2\lambda t}\,dt.
\]
By Laplace's method \cite[Chapter~4]{deBruijn1958}, there exists $C_{T,\kappa}>0$ such that, for $\lambda\geq1$,
\[
H_T(\lambda)^2\leq C_{T,\kappa}\lambda^{-3/4}e^{-2\sqrt{2\kappa}\,\mathcal{L}(\omega)\sqrt\lambda}.
\]
Given $0<\beta<1$, choose $\kappa\in(\beta^2/2,1/2)$. Then $\sqrt{2\kappa}>\beta$, and the polynomial factor is absorbed by the exponential gap. Hence \eqref{H-sharp} holds for $\lambda\geq1$; boundedness of $H_T$ extends it to $[0,1]$. \qed

\section{Applications}\label{sec:applications}
%%%%%%%%%%%%%%%%%%%%%%%%%%%%%%%%%%%%%

In this section we verify Assumptions~\ref{item:A1}--\ref{item:A3} in four geometric settings and draw the consequences of Theorem~\ref{thm:main_thm}.

The following table summarizes the verification of~\ref{item:A1}--\ref{item:A3} in the settings treated in this section. For each example we record the ultracontractivity rate $\alpha_1$ of Assumption~\ref{item:A1}, the pointwise Weyl exponent $\alpha_2$ and slowly varying correction $\chi_2$ of~\ref{item:A3}, the principal constant $c(x)$, the geometric scale $\mathcal{L}(\omega)$, and the key references. In all examples but the sub-Riemannian one (for concreteness we present the Grushin case), the ultracontractivity exponent $\alpha_1$ of~\ref{item:A1} coincides with $\alpha_2$; the sub-Riemannian case, where $\alpha_1>\alpha_2$, is discussed separately below.
For the compact geometric examples and an open observation set, the table displays the ordinary maximal distance; by \eqref{ess-dist}, it agrees with $\mathcal{L}(\omega)$. For a merely measurable observation set, $\mathcal{L}(\omega)$ always retains its essential meaning. The relevant metric is indicated explicitly in each row.

\medskip
\begin{center}
\renewcommand{\arraystretch}{1.4}
\begin{tabular}{lcccccl}
\toprule
\textbf{Setting} & $\alpha_1$ & $\alpha_2$ & $\chi_2$ & $c(x)$ & $\mathcal{L}(\omega)$ & \textbf{References} \\
\midrule
Compact Riem., $\nabla^*\nabla+V$ & $d/2$ & $d/2$ & $1$ & $\frac{r_\V\omega_d}{(2\pi)^d}\frac{d\mathrm{vol}_g}{d\mu}$ & $\displaystyle\sup_{x\in M}d_g(x,\omega)$ & \cite{Grigoryan2009_HeatKernel,sikora2004riesz} \\[3pt]
Schr\"odinger $-\Delta+V$, $\R^d$ & $d/2$ & $d/2$ & $1$ & $\frac{\omega_d}{(2\pi)^d}$ & $ \leq\frac{\sqrt{d}}{2}L$ & \cite{DaviesSimon1984} \\[3pt]
Grushin\ sub-Laplacian & $3/2$ & $1$ & $1$ & $c_{\rm sR}(x)$ & $\displaystyle\sup_{x\in M}d_{\rm sR}(x,\omega)$ & \cite{CdVHT_AHL_2021,SanchezCalle1984FundamentalSolutions} \\[3pt]
$\delta'$-graph, $-\Delta_G+q$ & $1/2$ & $1/2$ & $1$ & $\frac{1}{\pi}$ & $\displaystyle\sup_{x\in G}d_G(x,\omega)$ & \cite{BK24a,OdzakSceta2019WeylQuantumGraphs} \\
\bottomrule
\end{tabular}
\end{center}
\medskip

\noindent Here $\omega_d$ is the volume of the unit ball in $\R^d$, $L$ is the thickness parameter (see \eqref{thick-eq}), and $c_{\rm sR}$ depends on the sub-Riemannian structure (see \cite{CdVHT_AHL_2021,ColinDeVerdiereHillairetTrelat2022arXiv2212_02920} for explicit expressions). It is positive on the regular set and may be unbounded near the singular set, which is why~\ref{item:A3} does not require $c$ to be bounded (Remark~\ref{rem:A3-remarks}).
For the Grushin model, the uniform on-diagonal upper bound has exponent $Q_{\rm up}=3$, whereas the homogeneous dimension on the full-measure regular set is $Q_{\rm reg}=2$; hence $\alpha_1=3/2>\alpha_2=1$. This formulation does not presuppose a stratification of the singular set.

\subsection{Compact Riemannian manifolds and Laplace-type operators}\label{subsec:riem-app}

Let $(M,g,\mu)$ be a smooth connected compact weighted Riemannian manifold of dimension $d$, possibly with boundary $\partial M$, where $g$ is the Riemannian metric and $\mu$ is a smooth positive measure. We endow $M$ with the Riemannian distance. Let $\V\to M$ be a smooth metric vector bundle of rank $r_\V$ with metric connection $\nabla$, and let $V\in\Gamma(\End(\V))$ be smooth and self-adjoint. We assume that the standard realization of the Laplace-type operator $A=\nabla^*_\mu\nabla+V$ is nonnegative; pointwise nonnegativity of $V$ is a sufficient, but not necessary, condition. If $\partial M\neq\emptyset$, we take the Friedrichs realization with Dirichlet boundary conditions. Its closed quadratic form is $\mathcal{E}(u)=\int_M|\nabla u|^2\,d\mu+\int_M\langle Vu,u\rangle_{\V}\,d\mu$, the norm and inner product being the natural fiberwise ones.

The volume doubling property \eqref{framework-VD} holds since $M$ is compact and $\mu$ smooth. The principal symbol $\sigma_A(x,\xi)=|\xi|^2_{g_x}\mathrm{Id}_{\V_x}$ is of Laplace type, hence $A$ is elliptic.

\medskip\noindent\emph{Verification of~\ref{item:A1}.}\par\noindent Ultracontractivity is classical; see \cite{Grigoryan2009_HeatKernel,Ludewig2019StrongShortTime}. The on-diagonal estimate
$\mathrm{Tr}_{\V_x}(p_t(x,x))\leq C t^{-d/2}$
gives $\|e^{-tA}\|_{2\to\infty}\leq C t^{-d/4}$, hence $\alpha_1=d/2$. Here the heat kernel is smooth and, for every fixed $t>0$, has the usual eigenfunction expansion
$p_t(x,y)=\sum_{n\geq0}e^{-t\lambda_n}\varphi_n(x)\otimes\varphi_n(y)^\flat$.

\medskip\noindent\emph{Verification of~\ref{item:A2}.} The wave equation associated with a Laplace-type operator has propagation speed at most one, independently of the zeroth-order term. The equivalence between finite propagation and Davies--Gaffney estimates \cite{sikora2004riesz} therefore gives~\ref{item:A2}.

\medskip\noindent\emph{Verification of~\ref{item:A3}.} The H\"ormander local Weyl law \cite{Hormander1968Acta}, which depends only on the principal symbol and is of local nature, gives, at every interior point $x$ (a set of full measure), the pointwise principal term $E_x(\lambda)\sim c(x)\lambda^{d/2}$ with
$c(x)=\frac{r_\V\,\omega_d}{(2\pi)^d}\,\frac{d\mathrm{vol}_g}{d\mu}(x)>0$ and $\chi_2\equiv1$, so $\alpha_2=d/2=\alpha_1$.

\medskip All assumptions are satisfied, and Theorem~\ref{thm:main_thm} applies.

\paragraph{Laplace--Beltrami operator.} For $A=\Delta_g$ ($V=0$, $\V=M\times\R$) and every nontrivial open observation set $\omega\subset M$, Theorem~\ref{thm:main_thm} gives, for every $0<\kappa<1/2$, the estimate $h(t)\leq A_{T,\kappa}e^{-\kappa\mathcal{L}(\omega)^2/t}$ whenever $h\in\Obs([0,T])$ satisfies \eqref{gen-obs}; equivalently, $\limsup_{t\downarrow0}t\log h(t)\leq-\mathcal{L}(\omega)^2/2$.

For $T=+\infty$, Corollary~\ref{cor:conj_ervzua} gives $\gamma_\infty\geq\mathcal{L}(\omega)^2/2$ for the Dirichlet Laplacian when $\partial M\neq\emptyset$, and on $(\ker\Delta_g)^\perp$ when $M$ is closed. This resolves the maximal-distance problem from \cite{ErvedozaZuazua2011} with the optimal universal threshold and, in particular, covers smooth bounded domains in $\R^d$ with Dirichlet boundary conditions.

With $T<+\infty$, we recover the usual geometric bounds on the cost of null-controllability, and Corollary~\ref{cor:obs_sharp_cor} shows that the sharp observability inequality \eqref{sharpobs} is optimal within the integrated observability approach.

\paragraph{Coupled heat equations.} Let $\Omega\subset\R^d$ be a bounded connected domain with regular boundary, let $N\geq1$, and take $\V=\Omega\times\R^N$. For a constant symmetric nonnegative matrix $V\in\R^{N\times N}$, the Dirichlet realization of $A=-\Delta+V$ is a Laplace-type operator covered above. It governs the system
\[
\partial_tu_i-\Delta u_i+\sum_{j=1}^N V_i^ju_j=0\quad\text{in }(0,+\infty)\times\Omega,\qquad i=1,\ldots,N.
\]
Observability of coupled parabolic systems is studied, for instance, in \cite{LissyZuazua2019InternalObservabilityCoupledSystems,MazariFouquerPrivatTrelat2025LargeTimeOptimalObservation}. In the present symmetric constant-coefficient setting, when every component is observed on the same nonempty open set $\omega$, diagonalizing $V=P^\top\mathrm{diag}(\lambda_1,\ldots,\lambda_N)P$ reduces the problem to $N$ scalar heat equations with nonnegative constant potentials. Theorem~\ref{thm:main_thm} therefore gives the same barrier $\gamma\geq\mathcal{L}(\omega)^2/2$ for the finite- and infinite-time integrated constants. More generally, the Laplace-type formulation above includes Hodge Laplacians on differential forms (via the Weitzenb\"ock formula), even though the corresponding curvature endomorphism need not be pointwise nonnegative.

\subsection{\texorpdfstring{Schr\"odinger operators on $\R^d$}{Schrodinger operators on Euclidean space}}\label{subsec:schro-app}

We consider $X=\R^d$ with the Euclidean distance, $\mu$ the Lebesgue measure, and $A=-\Delta+V$ a Schr\"odinger operator on $L^2(\R^d)$, where $\Delta$ is the usual Laplacian and $V\in C^\infty(\R^d)\cap L^\infty(\R^d)$ is a nonnegative potential (the bounded case; the free Laplacian $V=0$ is included).
 The volume doubling property~\eqref{framework-VD} is immediate.
Since $V\geq 0$, $A$ is self-adjoint and nonnegative. Moreover, $e^{-tA}$ is dominated by the free heat semigroup via the Feynman--Kac formula (see \cite{DaviesSimon1984}): for all $f\in L^2(\R^d)\cap L^\infty(\R^d)$,
\[
\forall x\in\R^d,\quad |(e^{-tA}f)(x)|=\left |\E\Big[\exp\Big(-\int_0^t V(B_s^x)\,ds\Big) f(B_t^x)\Big]\right| \leq \E[|f(B_t^x)|] = (e^{t\Delta}|f|)(x),
\]
where $(B_t^x)_{t\geq0}$ is Brownian motion starting at $x$, normalized so that its generator is $\Delta$. The domination gives~\ref{item:A1}; applying it to data supported in $E$ and then using the free Davies--Gaffney estimate gives~\ref{item:A2}.
Assumption~\ref{item:A3} holds by the pointwise spectral asymptotics of H\"ormander \cite{Hormander1968Acta} for elliptic operators, whose proof is local and applies to $-\Delta+V$ on $\R^d$ with $V$ smooth and bounded: a bounded potential is a zeroth-order perturbation and the principal term is unchanged, $E_x(\lambda)\sim c_d\,\lambda^{d/2}$ with $c_d=\omega_d/(2\pi)^d$, hence $\alpha_2=d/2=\alpha_1$ and $\chi_2\equiv1$. Alternatively, since $0\leq V\leq\|V\|_{L^\infty}$, the Feynman--Kac formula gives the two-sided kernel bound $e^{-t\|V\|_{L^\infty}}\,p^0_t(x,x)\leq p_t(x,x)\leq p^0_t(x,x)=(4\pi t)^{-d/2}$, where $p^0$ denotes the free heat kernel; hence $p_t(x,x)\sim(4\pi t)^{-d/2}$ as $t\to0^+$, and Karamata's Tauberian theorem \cite[Theorem~1.7.1]{BinghamGoldieTeugels1987} yields the same asymptotics, at every $x\in\R^d$. For the free Laplacian $V=0$ everything is exact and uniform: $E_x(\lambda)=c_d\,\lambda^{d/2}$ independently of $x$.

\paragraph{Free Laplacian.} We start with $V=0$, so $A=-\Delta$. For the free heat equation on $\R^d$, observability from a measurable set $\omega$ is governed by a thickness (equidistribution) condition: there exist $L>0$ and $\rho>0$ such that
\begin{equation}\label{thick-eq}
\forall x\in\R^d,\quad|\omega\cap(x+L[0,1]^d)|\geq\rho L^d,
\end{equation}
see \cite{EgidiVeselic2018Sharp,WangWangZhangZhang2019Observable} for this sharp characterization; observability of the heat equation for an open $\omega$ in a bounded domain goes back to the earlier works \cite{FursikovImanuvilov1996,LebeauRobbiano1995CPDE}. Since $\mathcal{L}(\omega)\leq\frac{\sqrt{d}}{2}L$, Theorem~\ref{thm:main_thm} yields $\gamma \geq\frac{1}{2}\mathcal{L}(\omega)^2$. Using $\gamma=2\mathcal{K}_{\rm heat}$ to translate the control-norm estimate of \cite[Theorem~4.9]{NakicTauferTautenhahnVeselic2020_SharpEstimatesHomogenizationControlCost} into our squared-observability normalization, we obtain
\[\tfrac{1}{2}\mathcal{L}(\omega)^2\leq \gamma\leq C_1 d^2 L^2\ln^2(C_2^d/\rho),\]
where $C_1,C_2>0$ are universal constants. For standard periodic arrays whose cell size is comparable to the diameter of the largest hole, $L$ and $\mathcal{L}(\omega)$ are of the same order, so the two estimates have matching geometric scaling; no such comparison is asserted for an arbitrary thick set.

\paragraph{Bounded potential.} Now consider a general nonnegative $V\in C^\infty(\R^d)\cap L^\infty(\R^d)$. Under various additional assumptions on the potential, several works give geometric conditions on $\omega$ ensuring observability for the heat equation associated with $A=-\Delta+V$; see \cite{DickeSeelmannVeselic2024,DuanWangZhang2020,LebeauMoyano2019,NakicTauferTautenhahnVeselic2018,NakicTauferTautenhahnVeselic2020_SharpEstimatesHomogenizationControlCost}.
In particular, we say that $\omega\subset\R^d$ is $(G,\delta)$-equidistributed for some $G>0$ and $\delta\in(0,G/2)$ if
\[
\forall j\in(G\Z)^d,\ \exists z_j\in j+(-G/2,G/2)^d,\quad B(z_j,\delta)\subset\omega\cap(j+(-G/2,G/2)^d).
\]
For every point $x$ in a cell, each coordinate of $x-z_j$ has absolute value at most $G-\delta$; hence $\mathcal{L}(\omega)\leq\sqrt d\,(G-\delta)\leq\sqrt d\,G$. Using Theorem~\ref{thm:main_thm}, the identity $\gamma=2\mathcal{K}_{\rm heat}$, and \cite[Theorem~4.11]{NakicTauferTautenhahnVeselic2020_SharpEstimatesHomogenizationControlCost}, we obtain
$\tfrac{1}{2}\mathcal{L}(\omega)^2\leq \gamma\leq C_d G^2\ln^2(G/\delta)$,
where $C_d>0$ depends only on $d$.

\subsection{\texorpdfstring{Compact sub-Riemannian and Grushin-type structures}{Compact sub-Riemannian and Grushin-type structures}}\label{subsec:sR-app}

Let $M$ be a smooth connected compact manifold, possibly with boundary, equipped with a smooth positive measure $\mu$, and let $X_1,\ldots,X_m$ be smooth vector fields satisfying H\"ormander's bracket-generating condition. Their span is allowed to have nonconstant rank. We denote by $d_{\rm sR}$ the associated sub-Riemannian distance (see \cite{AgrachevBarilariBoscain2019}) and consider the sum-of-squares operator
\[
\Delta=\sum_{i=1}^m X_i^*X_i,
\]
where $X_i^*$ is the formal adjoint in $L^2(M,\mu)$. H\"ormander's theorem \cite{Hormander1967Acta} gives hypoellipticity and subelliptic estimates. We take the nonnegative self-adjoint realization of $\Delta$, with Friedrichs--Dirichlet boundary conditions when $\partial M\neq\emptyset$ (see \cite{CdVHT_AHL_2021}).

Let $M_{\rm reg}$ be the regular set, namely the open set of points at which the growth vector is locally constant. In the equiregular case, $M_{\rm reg}=M$. In the non-equiregular case considered here, we assume that $\mu(M\setminus M_{\rm reg})=0$ and that $Q(x)=Q_{\rm reg}$ for every $x\in M_{\rm reg}$, where $Q(x)$ is the local homogeneous dimension. This is only a full-measure regularity assumption; no stratification of the singular set is required. We also assume a uniform on-diagonal upper bound $p_t(x,x)\leq C t^{-Q_{\rm up}/2}$ for $0<t\leq1$, for some $Q_{\rm up}\geq Q_{\rm reg}$. It holds with $Q_{\rm up}=Q_{\rm reg}$ in the compact equiregular case, and with $(Q_{\rm up},Q_{\rm reg})=(3,2)$ for the classical Grushin model.

The volume doubling property holds for the compact bracket-generating structures under consideration (see \cite{AgrachevBarilariBoscain2019,SanchezCalle1984FundamentalSolutions}). The heat semigroup has a smooth kernel, and the preceding uniform bound gives~\ref{item:A1} with $\alpha_1=Q_{\rm up}/2$. Finite propagation for the sub-Riemannian wave equation was proved in \cite{Melrose1986PropagationWaveGroupSubelliptic}; by \cite{sikora2004riesz}, this is equivalent to~\ref{item:A2}.

For every $x\in M_{\rm reg}$, the local on-diagonal asymptotics $p_t(x,x)\sim c_0(x)t^{-Q_{\rm reg}/2}$ follow from \cite{CdVHT_AHL_2021} (see also \cite{ColinDeVerdiereHillairetTrelat2022arXiv2212_02920}). Karamata's Tauberian theorem \cite[Theorem~1.7.1]{BinghamGoldieTeugels1987} then gives $E_x(\lambda)\sim c(x)\lambda^{Q_{\rm reg}/2}$. Hence~\ref{item:A3} holds with $\alpha_2=Q_{\rm reg}/2$ and $\chi_2\equiv1$. The coefficient $c$ is positive on $M_{\rm reg}$ and may diverge near the singular set, which is allowed by Assumption~\ref{item:A3}. In the Grushin case, this yields $\alpha_1=3/2>\alpha_2=1$.

Under these assumptions, Theorem~\ref{thm:main_thm} applies. The controllability theory of Grushin heat equations is notably sensitive to the observation geometry. For the classical operator $G=-\partial_x^2-x^2\partial_y^2$ on $(-1,1)\times(0,1)$, observed from a vertical strip $(a,b)\times(0,1)$ with $0<a<b<1$, a positive minimal control time occurs \cite{BeauchardCannarsaGuglielmi2014JEMS}; minimal-time phenomena, controllable data and more general nonrectangular observation regions are further studied in \cite{BeauchardDardeErvedoza2020AIF,BeauchardMillerMorancey2015JDE,DuprezKoenig2020COCV}. Those works determine whether observability holds for a given geometry and horizon. Our theorem addresses the complementary issue of the unavoidable weight once an integrated inequality does hold: for every $0<\kappa<1/2$, it forces
\[
h(t)\leq A_{T,\kappa}e^{-\kappa\mathcal{L}(\omega)^2/t},
\]
equivalently $\limsup_{t\downarrow0}t\log h(t)\leq-\mathcal{L}(\omega)^2/2$, where $\mathcal{L}(\omega)$ is computed with $d_{\rm sR}$ (the Grushin distance in the model above). In particular, any exponential weight $e^{-\widetilde\gamma/t}$ must satisfy $\widetilde\gamma\geq\mathcal{L}(\omega)^2/2$; equivalently, the associated fast-control rate in the usual control-norm convention is at least $\mathcal{L}(\omega)^2/4$.

\subsection{\texorpdfstring{Compact metric graphs with $\delta'$-coupling}{Compact metric graphs with delta-prime coupling}}\label{subsec:graphs-app}

Let $G=(E,V)$ be a connected finite compact metric graph with edges of lengths $\ell_e$, total length $L=\sum_{e}\ell_e$, endowed with its path distance $d_G$ and Lebesgue measure $\mu$. Let $q\in L^\infty(G)$ with $q\geq 0$ smooth edgewise, and $\beta=(\beta_v)_{v\in V}\in(0,+\infty)^{|V|}$. We consider $A=-\Delta_G+q$ with $\delta'$-coupling conditions at each vertex $v$:
the normal derivative takes a common value $f'(v)$ at $v$ (i.e.\ the inward derivatives $\partial_n f_e(v)$ along the edges $e\sim v$ all agree), while the function values satisfy $\sum_{e\sim v}f_e(v)=\beta_v f'(v)$ (so $f$ itself may jump across $v$).
 As detailed in \cite{BerkolaikoKuchment2013,BK24a}, $A$ is the self-adjoint operator associated with the quadratic form
\[\mathfrak{a}_\beta[f]=\int_G(|f'|^2+q|f|^2)\,d\mu+\sum_{v\in V}\frac{1}{\beta_v}\Bigl|\sum_{e\sim v}f_e(v)\Bigr|^2,\qquad D(\mathfrak{a}_\beta)=\bigoplus_{e\in E}H^1(e).\]

The doubling property holds because $G$ is compact and one-dimensional. Since $D(\mathfrak a_\beta)=\bigoplus_{e\in E}H^1(e)$, the one-dimensional Sobolev inequality on the finite collection of edges, combined with the standard analytic-semigroup energy estimate, gives $\|e^{-tA}\|_{2\to\infty}\leq C t^{-1/4}$ for $0<t\leq1$; hence~\ref{item:A1} holds with $\alpha_1=1/2$. By Proposition~\ref{prop:heat-kernel}, the heat kernel $p_t\in L^\infty_\Delta(G\times G)$ exists. Since $G$ is compact and $A$ has compact resolvent, there is an orthonormal eigenbasis $(\varphi_n)_{n\geq1}$, ordered so that $0\leq\lambda_1\leq\lambda_2\leq\cdots$ and repeated according to multiplicity. The spectral theorem gives $p_t(x,y)=\sum_{n\geq1}e^{-t\lambda_n}\varphi_n(x)\varphi_n(y)$ in $L^2$ and edgewise away from the vertices; the classical continuous-kernel form of Mercer's theorem is unavailable because $p_t$ need not be continuous.
By \cite{BolteEggerRueckriemen2015HeatKernelResolventMetricGraphs}, $p_t$ is smooth edgewise (on each edge up to endpoints), but it is \emph{not continuous on $G\times G$} in general: the $\delta'$-coupling conditions allow discontinuities of $\varphi_n$ at vertices, so the semigroup $e^{-tA}$ is not Feller. This is exactly the reason why we work with $L^\infty_\Delta$ rather than continuous kernels.

Assumption~\ref{item:A2} follows from the standard Davies Lipschitz-weight argument \cite{Davies1989}. Indeed, multiplication by $e^\varphi$, with $\varphi$ Lipschitz and continuous on the metric graph, preserves $D(\mathfrak a_\beta)$. In the polarized conjugated form $\mathfrak a_\beta(e^{-\varphi}f,e^{\varphi}f)$, the factors $e^{-\varphi(v)}$ and $e^{\varphi(v)}$ cancel in each vertex term because all incident traces of $\varphi$ have the same value at $v$; hence the vertex contribution creates no commutator. The edge contribution then gives the usual bound by $\|\varphi'\|_\infty^2$.

\medskip\noindent\emph{Verification of~\ref{item:A3}.} By \cite[Proposition~8]{BK24a}, at every interior point $x$ of an edge,
$\frac{1}{N}\sum_{n=1}^N|\varphi_n(x)|^2\longrightarrow\frac{1}{L}$.
Consequently, with $N(\lambda)=\#\{n\mid\lambda_n\leq\lambda\}$,
$E_x(\lambda)=\sum_{\lambda_n\leq\lambda}|\varphi_n(x)|^2\sim\frac{N(\lambda)}{L}$.
The Weyl counting law $N(\lambda)\sim(L/\pi)\sqrt\lambda$ \cite{BerkolaikoKuchment2013,OdzakSceta2019WeylQuantumGraphs} therefore yields $E_x(\lambda)\sim\sqrt\lambda/\pi$. Since the vertex set is negligible, Assumption~\ref{item:A3} holds with $\alpha_2=1/2$, $c\equiv1/\pi$ and $\chi_2\equiv1$; thus $\alpha_1=\alpha_2=1/2$ in this example.

Therefore all results of Section~\ref{subsec:main-results} apply. To the best of our knowledge, heat observability for $\delta'$-coupled graph Laplacians has not been treated specifically. Related results concern parabolic equations on loops with other vertex conditions \cite{ApraizBarcenaPetisco2023ParabolicLoops} and wave equations with $\delta'$-coupling \cite{AmmariDucaJolyLeBalch2025GGCC,AvdoninEdwardLeugering2023CycleDeltaPrime}.

%\appendix
%\chapter{Appendix}
\appendix
\numberwithin{thm}{section}
\numberwithin{lem}{section}
\numberwithin{prop}{section}
\numberwithin{cor}{section}
\numberwithin{defn}{section}

\section{Appendix: pointwise spectral measures and weak wave kernels}\label{app:kernel-toolbox}

In this appendix we develop the functional-analytic framework for integral kernels on our metric measure space. The underlying ingredients are classical in concrete smooth settings (see \cite{Grigoryan2009_HeatKernel,Hormander1968Acta} for elliptic operators on compact manifolds), but the vector-bundle formulation for kernels defined $\Delta$-a.e.\ and the pointwise spectral-measure construction are included to make the argument rigorous without compact resolvent or kernel continuity.

Related kernel constructions on metric measure spaces are developed, for instance, in \cite{Davies1989,DaviesSimon1984,GrigoryanHu2014DV}, often in a Dirichlet-form or positivity-preserving setting. The statements below are tailored to the present vector-bundle framework, to kernels defined only $\Delta$-a.e., and to operators that need not have compact resolvent or arise from a regular strongly local Dirichlet form.

\subsection{\texorpdfstring{Kernels and $L^\infty_\Delta$ spaces}{Kernels and L-infinity-Delta spaces}}\label{app:subsec-Linfty}

In order to talk about integral operators acting on sections $\V\to X$, we introduce the external tensor product bundle
$\Vbox=\pi_1^*\V\otimes\pi_2^*(\V^*)\to X\times X$,
where $\pi_1,\pi_2:X\times X\to X$ are the projections. For $(x,y)\in X\times X$, $(\Vbox)_{(x,y)}\simeq\V_x\otimes\V_y^*\simeq\mathcal{L}(\V_y,\V_x)$, endowed with Hilbert--Schmidt norm $\HSnorm{\cdot}$.

An operator $T\in\mathcal{L}(L^p(X;\V),L^q(X;\V))$ is an \emph{integral operator} if there exists a measurable section $K_T:X\times X\to\Vbox$ such that
\[(Tf)(x)=\int_X K_T(x,y)f(y)\,d\mu(y) \quad \text{for a.e.\ } x.\]

\begin{defn}[$\Delta$-a.e.\ and $L^\infty_\Delta$]\label{def:Delta-ae}
Let $f,g\in\mathcal{B}(X\times X)$. We write $f=g$ \emph{$\Delta$-a.e.} if there exists a $\mu$-null set $N\subset X$ with $f(x,y)=g(x,y)$ for all $(x,y)\in(X\setminus N)^2$. We say $f$ is \emph{$\Delta$-essentially bounded} if $|f(x,y)|\leq C$ on $(X\setminus N)^2$ for some $C\geq 0$ and $\mu$-null $N$. The space $L^\infty_\Delta(X\times X)$ is the quotient by $\Delta$-a.e.\ equality, with norm
\[\|[f]\|_{L^\infty_\Delta}=\inf\{C\geq 0:\exists N,\,\mu(N)=0,\,|f|\leq C\text{ on }(X\setminus N)^2\}.\]
\end{defn}

The $\Delta$-a.e.\ relation is strictly stronger than $\mu\otimes\mu$-a.e.: there is a natural (non-injective in general) map $L^\infty_\Delta\to L^\infty(X\times X)$. The key advantage is that for $f\in L^\infty_\Delta$, the diagonal value $x\mapsto f(x,x)$ is well defined for $\mu$-a.e.\ $x$, with $\|f(\cdot,\cdot)\|_\infty\leq\|f\|_{L^\infty_\Delta}$.

This extends to bundle-valued kernels: $L^\infty_\Delta(X\times X;\Vbox)$ consists of $\Delta$-essentially bounded measurable sections modulo $\Delta$-a.e.\ equality. For $K\in L^\infty_\Delta(X\times X;\Vbox)$, the diagonal $x\mapsto K(x,x)\in\End(\V_x)$ is well defined $\mu$-a.e., and $x\mapsto\tr_{\V_x}(K(x,x))\in L^\infty(X)$.

For dealing with distribution-valued kernels (the wave kernel), let $W$ be a Hausdorff locally convex topological vector space with dual $W'$. We define $L^\infty_\Delta(X\times X;W'\otimes(\Vbox))$ as sections $F$ such that $\langle F,w\rangle\in L^\infty_\Delta(X\times X;\Vbox)$ for all $w\in W$, modulo the \emph{weak-$*$ $\Delta$-a.e.} relation: $F=G$ iff $\langle F,w\rangle=\langle G,w\rangle$ $\Delta$-a.e.\ for all $w$.
When $W$ is separable, a countable-dense argument shows:

\begin{lem}\label{lem:delta-ae}
If $W$ is separable, $F=G$ weak-$*$ $\Delta$-a.e.\ iff there exists a $\mu$-null set $N\subset X$ such that $F(x,y)=G(x,y)$ in $W'\otimes \mathcal{L}(\V_y,\V_x) \simeq \mathcal{L}(W,\mathcal{L}(\V_y,\V_x))$ for all $(x,y)\in(X\setminus N)^2$.
\end{lem}
\begin{proof}
Pick a countable dense $\{w_k\}\subset W$. For each $k$, $\langle F,w_k\rangle=\langle G,w_k\rangle$ on $(X \setminus N_k)^2$ for a $\mu$-null set $N_k$. Set $N:=\bigcup_k N_k$.
On $(X\setminus N)^2$, density of $\{w_k\}$ together with the weak-$*$ separation of points yields $F(x,y)=G(x,y)$ in $W'\otimes\mathcal{L}(\V_y,\V_x)$. The converse is immediate.
\end{proof}
\subsection{Heat kernel and pointwise spectral measures}\label{subsec:Ultra-spec}

We start by establishing a few properties that will ensure that everything is well defined in our framework. From now on, the almost everywhere relation on $X$ will be understood as $\mu$-almost everywhere.

\subsubsection{Heat kernel and ultracontractivity} We first show that Assumption~\ref{item:A1} yields a heat kernel associated with $A$; a converse statement is recorded in Remark~\ref{rem:hk-conv}.

\begin{prop}\label{prop:heat-kernel}
    Under Assumption~\ref{item:A1}, for all $t>0$ there exists a unique $p_t \in L^\infty_\Delta(X \times X;\Vbox)$ that satisfies the following properties:
    \begin{enumerate}[label=(\roman*)]
        \item\label{it:heat-ker} For all $f \in L^2(X;\V)$ we have for a.e.\ $x\in X$
        \begin{equation}\label{heat-ker}
        (e^{-tA} f)(x) = \int_X p_t(x,y) f(y)\, d\mu(y).
        \end{equation}
        \item\label{it:sym-ker} We have $p_t(x,y) = p_t(y,x)^{*}$ for $\Delta$-a.e. $(x,y) \in X \times X$.
        \item\label{it:sg-id} Let $s>0$, we have the semigroup identity for $\Delta$-a.e. $(x,y) \in X \times X$
        \begin{equation}\label{semigroup_kernel_bundle}
             p_{t+s}(x,y) = \int_X p_t(x,z) p_s(z,y)\, d\mu(z).
        \end{equation}
        In particular, for every $t>0$ and for almost every $x\in X$,
        \begin{equation}\label{diag_trace_vs_L2}
            \tr_{\V_x}\!\big(p_{2t}(x,x)\big) = \int_X \HSnorm{p_{t}(y,x)}^2\, d\mu(y)
            \leq r_\V \,\|e^{-tA}\|_{2 \rightarrow \infty}^2.
        \end{equation}
    \end{enumerate}
\end{prop}

\begin{proof}

\emph{Existence.} Fix $t>0$. Since $e^{-tA/2}\in\mathcal{L}(L^2(X;\V),L^\infty(X;\V))$, separability and the Riesz representation theorem provide, outside a null set $N_t$, bounded evaluation maps $R_{t/2,x}:L^2(X;\V)\to\V_x$. Their adjoints take values in $L^2(X;\V)$, and
$p_t(x,y):=R_{t/2,x}R_{t/2,y}^*$, for $x,y\notin N_t$,
is a bounded measurable kernel of $e^{-tA}$ on the full-measure square $(X\setminus N_t)^2$. This formula gives the symmetry. For fixed $s,t>0$, composing the corresponding rows and columns gives a kernel of $e^{-(t+s)A}$; kernel uniqueness therefore yields the semigroup identity $\Delta$-a.e. Applying that identity with $s=t$, together with symmetry, gives \eqref{diag_trace_vs_L2} for each fixed $t$. Exceptional sets may be unified whenever finitely or countably many times are involved; this is the only simultaneous version used below.

\smallskip
\emph{Uniqueness.} Now if there exists $\{\tilde p_t\}_{t>0}$ also satisfying~\ref{it:heat-ker}--\ref{it:sg-id}, then for each $t>0$, thanks to~\ref{it:heat-ker}, a separability argument gives a null set $M_t$ with $\tilde p_t(x,\cdot)=p_t(x,\cdot)$ a.e.\ for $x\notin M_t$. The semigroup identity and symmetry for both kernels give, for $x,y\notin M_{t/2}$:
$\tilde p_t(x,y)=\int\tilde p_{t/2}(x,z)\tilde p_{t/2}(z,y)\,d\mu(z)=\int p_{t/2}(x,z)p_{t/2}(z,y)\,d\mu(z)=p_t(x,y)$,
so $\tilde p_t=p_t$ $\Delta$-a.e.
\end{proof}

\begin{remark}\label{rem:hk-conv}
Conversely, suppose that $\{\tilde p_t\}_{t>0}\subset L^\infty_\Delta$ satisfies~\ref{it:heat-ker}--\ref{it:sg-id} and that $\tr_{\V_x}(\tilde p_{2t}(x,x))\leq Ct^{-\alpha_1}$ for $t\in(0,1]$ and a.e.\ $x$. Then Cauchy--Schwarz and the semigroup identity give $|(e^{-tA}f)(x)|\leq\tr_{\V_x}(\tilde p_{2t}(x,x))^{1/2}\|f\|_2$, hence Assumption~\ref{item:A1}.
\end{remark}

\subsubsection{Pointwise spectral measures} Our goal is to define integral kernels for compactly supported functions of $A$.
The purpose of this construction is twofold: it gives a pointwise Plancherel identity for the columns of kernels, and it provides the spectral representation needed to define weak wave kernels.
On a compact Riemannian manifold, smooth eigenfunctions and compact resolvent give the pointwise spectral expansion $K_{\phi(A)}(x,y)=\sum_n\phi(\lambda_n)\varphi_n(x)\otimes\varphi_n(y)^\flat$, suggesting the measure-valued kernel
\[
\Pi_{x,y} = \sum_n \delta_{\lambda_n}\, \varphi_n(x)\otimes \varphi_n(y)^\flat,
\]
so that $K_{\phi(A)}(x,y) = \int \phi(\lambda)\,d\Pi_{x,y}(\lambda)$. We extend this notion to our general framework. Since $A$ may lack compact resolvent or heat kernel continuity, we proceed differently.

Let $t>0$ and $x\in X$. As in Proposition~\ref{prop:heat-kernel}, we consider the Hilbert--Schmidt operator
\begin{equation*}
P_{t,x}:\V_x\to L^2(X;\V),\qquad (P_{t,x}u)(y)=p_t(y,x)u.
\end{equation*}
By \eqref{heat-ker} and the symmetry $p_t(x,y)=p_t(y,x)^*$, its adjoint satisfies for a.e.\ $x$ and all $g\in L^2(X;\V)$,
\begin{equation}\label{PtAdjoint}
P_{t,x}^*g=\int_X p_t(x,y)g(y)\,d\mu(y)=(e^{-tA}g)(x)\in \V_x.
\end{equation}
We introduce the following notation for the bundle of $\Vbox$-valued locally finite measures on $[0,+\infty)$
\begin{equation*}
\Mloc\big([0,+\infty);\Vbox\big) = C_c([0,+\infty))' \otimes (\Vbox).
\end{equation*}
Now we are able to state and prove the existence of the pointwise spectral measure of $A$ in our framework.

\begin{prop}\label{prop:spectral-mes}
There exists a unique $\Pi\in L^\infty_\Delta(X\times X;\Mloc([0,+\infty);\Vbox))$ such that for every bounded Borel $B\subset[0,+\infty)$ and $t>0$:
\begin{equation}\label{prop-specmes}
    \Pi_{x,y}(B)=P_{t/2,x}^*\,1_B(A)e^{tA}\,P_{t/2,y}\in\Hom(\V_y,\V_x)\quad\Delta\text{-a.e.},
\end{equation}
and the right-hand side is independent of $t$. Furthermore:
\begin{enumerate}[label=(\roman*)]
        \item For $\phi:[0,+\infty)\to\R$ bounded measurable with compact support, $\phi(A)$ is an integral operator with kernel
        \begin{equation}\label{kernel_phiA_by_E}
            K_{\phi(A)}(x,y)=\int_0^{+\infty}\phi(\lambda)\,d\Pi_{x,y}(\lambda)\in L^\infty_\Delta(X\times X;\Vbox),
        \end{equation}
        and defining $\nu_x=\tr_{\V_x}(\Pi_{x,x})$, we have for a.e.\ $x$:
        \begin{equation}\label{HS_L2_identity_app}
        \int_X\HSnorm{K_{\phi(A)}(x,y)}^2\,d\mu(y)=\int_0^{+\infty}\phi(\lambda)^2\,d\nu_x(\lambda).
        \end{equation}
        \item For almost every $x \in X$, $\nu_x$ is a nonnegative $\sigma$-finite measure on $[0,+\infty)$. Moreover, for $\Delta$-a.e. $(x,y) \in X \times X$, we denote by $|\Pi_{x,y}|_{\mathrm{HS}}$ the total variation measure of $\Pi_{x,y}$ induced by $\HSnorm{\cdot}$, then for every bounded Borel set $B \subset [0,+\infty)$
        \begin{equation}\label{TV_bound_bundle}
            |\Pi_{x,y}|_{\mathrm{HS}}(B) \leq \nu_x^{1/2}(B)\,  \nu_y^{1/2}(B).
        \end{equation}
        Consequently, for every nonnegative Borel function $q$ on $[0,+\infty)$, with the convention that the two sides may take the value $+\infty$,
        \begin{equation}\label{weighted-TV-bound}
        \int_0^{+\infty}q(\lambda)\,d|\Pi_{x,y}|_{\mathrm{HS}}(\lambda)
        \leq
        \left(\int_0^{+\infty}q(\lambda)\,d\nu_x(\lambda)\right)^{1/2}
        \left(\int_0^{+\infty}q(\lambda)\,d\nu_y(\lambda)\right)^{1/2}.
        \end{equation}

    \end{enumerate}
\end{prop}

\begin{proof}
Fix $t>0$ and a bounded interval $I\subset[0,+\infty)$.
For $x,y\in X$ (outside a negligible set, we fix a representative for $p_t$ as in the proof of Proposition~\ref{prop:heat-kernel}) and for Borel sets $B\subset I$, define
\begin{equation*}
\Pi_{x,y}(B)=P_{t/2,x}^*\;1_B(A)e^{tA}\;P_{t/2,y}\in \Hom(\V_y,\V_x).
\end{equation*}
Since $\lambda\mapsto 1_B(\lambda)e^{t\lambda}$ is bounded on $I$, the operator $1_B(A)e^{tA}$ is bounded on $L^2(X;\V)$, hence $\Pi_{x,y}(B)$ is well-defined.
Let $(B_n)_{n\ge0}$ be a sequence of disjoint Borel subsets of $I$. Since
$1_{\cup_n B_n}(\lambda)e^{t\lambda}=\sum_n 1_{B_n}(\lambda)e^{t\lambda}$ for all $\lambda\in I$, spectral calculus gives
\begin{equation*}
1_{\cup_n B_n}(A)e^{tA}=\sum_{n\ge0}1_{B_n}(A)e^{tA}
\quad\text{(strongly on }L^2(X;\V)\text{)}.
\end{equation*}
Therefore $\Pi_{x,y}\Big(\bigcup_{n\ge0}B_n\Big)=\sum_{n\ge0}\Pi_{x,y}(B_n)$,
so for each $(x,y)$, $\Pi_{x,y}$ is a countably additive $\Hom(\V_y,\V_x)$-valued measure on $I$. Because the same formula is used on one fixed full-measure square for every bounded Borel set, the measures obtained on bounded intervals are compatible under restriction and define a locally finite measure on $[0,+\infty)$. Weak measurability follows from the measurability of the heat-kernel columns and a monotone-class argument.
To obtain $\Delta$-essential boundedness, fix $\phi\in C_c([0,+\infty))$ and consider the bounded operator $\phi(A)e^{tA}$.
Then for $\Delta$-a.e.\ $(x,y)$,
\[
\HSnorm{\int_0^{+\infty}\phi(\lambda)\,d\Pi_{x,y}(\lambda)}
=\HSnorm{P_{t/2,x}^*\phi(A)e^{tA}P_{t/2,y}}
\le \|P_{t/2,x}\|_{\mathrm{HS}}\,\|\phi(A)e^{tA}\|\,\|P_{t/2,y}\|_{\mathrm{HS}}.
\]
By \eqref{diag_trace_vs_L2} applied at time $t/2$, $\|P_{t/2,x}\|_{\mathrm{HS}}^2=\tr_{\V_x}(p_t(x,x))$ is bounded a.e., hence the above estimate yields the required boundedness.

Now let $\phi:[0,+\infty)\to\R$ be bounded measurable with compact support. Define
$K(x,y)=\int_0^{+\infty}\phi(\lambda)\,d\Pi_{x,y}(\lambda)$.
Then $K\in L^\infty_\Delta(X\times X;\Vbox)$.
We show that $K$ is the kernel of $\phi(A)$.
Fix $f\in L^2(X;\V)$ and $x$ outside a null set where the kernel formula holds.
For any $u\in\V_x$, by definition of adjoints, Fubini, and \eqref{PtAdjoint},
\[
\Big\langle \int_X K(x,y)f(y)\,d\mu(y),\,u\Big\rangle_{\V_x}
=\int_X \langle f(y),K(x,y)^*u\rangle_{\V_y}\,d\mu(y)
=\scalp{f}{\,y\mapsto K(x,y)^*u\,}_{L^2(X;\V)}.
\]
Using $K(x,y)=P_{t/2,x}^*\phi(A)e^{tA}P_{t/2,y}$ and functional calculus,
$K(x,y)^*=P_{t/2,y}^*e^{tA}\phi(A)P_{t/2,x}$.
Hence, as a section in $y$,
\begin{equation}\label{Kstar_as_section}
y\mapsto K(x,y)^*u
=P_{t/2,y}^*e^{tA}\phi(A)P_{t/2,x}u
=\big(e^{-t/2A}e^{tA}\phi(A)P_{t/2,x}u\big)(y)
=\big(e^{t/2A}\phi(A)P_{t/2,x}u\big)(y).
\end{equation}
Therefore
\begin{multline*}
\Big\langle \int_X K(x,y)f(y)\,d\mu(y),\,u\Big\rangle_{\V_x}
=\scalp{f}{e^{t/2A}\phi(A)P_{t/2,x}u}_{L^2(X;\V)} \\
=\scalp{e^{t/2A}\phi(A)f}{P_{t/2,x}u}_{L^2(X;\V)}
=\langle P_{t/2,x}^*e^{t/2A}\phi(A)f,\,u\rangle_{\V_x}.
\end{multline*}
Finally, by spectral calculus,
\begin{equation*}
P_{t/2,x}^*e^{t/2A}\phi(A)f=(e^{-t/2A}e^{t/2A}\phi(A)f)(x)=(\phi(A)f)(x),
\end{equation*}
hence $\int_X K(x,y)f(y)\,d\mu(y)=(\phi(A)f)(x)$ for a.e.\ $x$.
This proves \eqref{kernel_phiA_by_E} and in particular the independence of \eqref{prop-specmes} from $t$.
We now prove \eqref{HS_L2_identity_app}.
Fix $x$ and choose an orthonormal basis $(u_\ell)_{\ell=1}^r$ of $\V_x$.
For each $\ell$, define the section
$s_\ell(y)=K(x,y)^*u_\ell\in \V_y$.
Then
\[
\sum_{\ell=1}^r\|s_\ell\|_{L^2(X;\V)}^2
=\int_X \sum_{\ell=1}^r |K(x,y)^*u_\ell|_{\V_y}^2\,d\mu(y)
=\int_X \HSnorm{K(x,y)}^2\,d\mu(y).
\]
On the other hand, by \eqref{Kstar_as_section} we have $s_\ell=e^{t/2A}\phi(A)P_{t/2,x}u_\ell$, hence
\begin{multline*}
\sum_{\ell=1}^r\|s_\ell\|_{L^2(X;\V)}^2
=\sum_{\ell=1}^r \scalp{e^{t/2A}\phi(A)P_{t/2,x}u_\ell}{e^{t/2A}\phi(A)P_{t/2,x}u_\ell}_{L^2(X;\V)} \\
=\sum_{\ell=1}^r \scalp{P_{t/2,x}u_\ell}{\phi(A)^2e^{tA}P_{t/2,x}u_\ell}_{L^2(X;\V)}
=\tr_{\V_x}\!\big(P_{t/2,x}^*\phi(A)^2e^{tA}P_{t/2,x}\big).
\end{multline*}
By definition of $\Pi_{x,x}$, this equals $\int_0^{+\infty}\phi(\lambda)^2\,d\nu_x(\lambda)$, proving \eqref{HS_L2_identity_app}.

Finally, we prove \eqref{TV_bound_bundle}.
Fix a bounded Borel set $B\subset[0,+\infty)$ and set $S_B=1_B(A)e^{tA}$, which is bounded and nonnegative self-adjoint on $L^2(X;\V)$.
Let $(e_i)_{i=1}^r$ be an orthonormal basis of $\V_y$ and $(f_j)_{j=1}^r$ an orthonormal basis of $\V_x$.
Then, using $S_B^{1/2}$ and Cauchy--Schwarz in $L^2(X;\V)$,
\[
\langle \Pi_{x,y}(B)e_i,f_j\rangle_{\V_x}
=\scalp{S_B P_{t/2,y}e_i}{P_{t/2,x}f_j}_{L^2(X;\V)}
=\scalp{S_B^{1/2}P_{t/2,y}e_i}{S_B^{1/2}P_{t/2,x}f_j}_{L^2(X;\V)}.
\]
Hence
\begin{align*}
\HSnorm{\Pi_{x,y}(B)}^2
&=\sum_{i=1}^r\sum_{j=1}^r |\langle \Pi_{x,y}(B)e_i,f_j\rangle_{\V_x}|^2 \\
&\le \sum_{i=1}^r\sum_{j=1}^r \|S_B^{1/2}P_{t/2,y}e_i\|_{L^2(X;\V)}^2\,\|S_B^{1/2}P_{t/2,x}f_j\|_{L^2(X;\V)}^2 \\
&=\tr_{\V_y}\!\big(P_{t/2,y}^*S_BP_{t/2,y}\big)\, \tr_{\V_x}\!\big(P_{t/2,x}^*S_BP_{t/2,x}\big)
=\nu_y(B)\,\nu_x(B).
\end{align*}
Thus $\HSnorm{\Pi_{x,y}(B)}\le \nu_x(B)^{1/2}\nu_y(B)^{1/2}$.
Now for a finite partition $(B_n)_{n=1}^N$ of $B$,
\begin{multline*}
\sum_{n=1}^N \HSnorm{\Pi_{x,y}(B_n)}
\le \sum_{n=1}^N \nu_x(B_n)^{1/2}\nu_y(B_n)^{1/2} \\
\le \Big(\sum_{n=1}^N \nu_x(B_n)\Big)^{1/2}\Big(\sum_{n=1}^N \nu_y(B_n)\Big)^{1/2}
=\nu_x(B)^{1/2}\nu_y(B)^{1/2}.
\end{multline*}
Taking the supremum over partitions gives \eqref{TV_bound_bundle}.

For completeness, \eqref{weighted-TV-bound} follows first for a nonnegative simple function $q=\sum_{j=1}^N a_j1_{B_j}$, with the $B_j$ disjoint, since \eqref{TV_bound_bundle} and Cauchy--Schwarz give
\[
\int q\,d|\Pi_{x,y}|_{\mathrm{HS}} \leq\sum_{j=1}^Na_j\nu_x(B_j)^{1/2}\nu_y(B_j)^{1/2} \leq\left(\int q\,d\nu_x\right)^{1/2} \left(\int q\,d\nu_y\right)^{1/2}.
\]
The general case follows by monotone approximation.

Finally, we prove uniqueness. If $\Pi'$ also satisfies \eqref{prop-specmes}, then for every interval $B$ with rational endpoints one has $\Pi'_{x,y}(B)=\Pi_{x,y}(B)$ $\Delta$-a.e.; since these intervals form a countable family generating the Borel $\sigma$-algebra of $[0,+\infty)$ and both measures are locally finite, Lemma~\ref{lem:delta-ae} (applied with $W=C_c([0,+\infty))$, which is separable) yields $\Pi'=\Pi$ in $L^\infty_\Delta\big(X\times X;\Mloc([0,+\infty);\Vbox)\big)$.
\end{proof}

Coming back to our general setting, one might wonder why we introduced the pointwise spectral measures $\Pi_{x,y}$ instead of directly saying that, for a bounded compactly supported Borel function $\phi : [0,+\infty) \rightarrow \mathbb{R}$, the integral kernel of $\phi(A)$ is given by
\begin{equation*}
    \Delta\text{-a.e. } (x,y) \in X \times X, \quad K_{\phi(A)}(x,y) = P_{t/2,x}^*\phi(A) e^{t A} P_{t/2,y}.
\end{equation*}
The reason is that when $\phi$ is not compactly supported, the above formula might not make sense, while the integral with respect to the pointwise spectral measure will still make sense thanks to the quantitative estimate in Assumption~\ref{item:A1}. In order to do so, we recall the diagonal spectral function
\begin{equation*}
    E_{x}(\lambda) = \tr_{\V_x}(K_{1_{[0,\lambda]}(A)}(x,x)) = \nu_x([0,\lambda]),
\end{equation*}
and everything is made clear in the following proposition.
\begin{prop}\label{prop:kernel-weyl}
 Let $\phi : [0,+\infty) \rightarrow \mathbb{R}$ be a measurable function, such that there exist $C>0$ and $\varepsilon>0$ satisfying
        \begin{equation*}
            \forall \lambda \geq 0, \quad |\phi(\lambda)| \leq \frac{C}{(1+\lambda)^{\alpha_1 + \varepsilon}},
        \end{equation*}
        where $\alpha_1$ comes from~\ref{item:A1}.
        Then $\phi(A)$ is an integral operator with kernel $K_{\phi(A)} \in L^\infty_\Delta(X\times X;\Vbox)$ given by
        \begin{equation}\label{kernel_loc_weyl}
             \Delta\text{-a.e. } (x,y) \in X \times X, \quad K_{\phi(A)}(x,y) = \int_{0}^{+\infty} \phi(\lambda)\, d\Pi_{x,y}(\lambda),
        \end{equation}
        and also, for almost every $x \in X$, we have
        \begin{equation}\label{L2rowHS}
            \int_X \HSnorm{K_{\phi(A)}(x,y)}^2\,d\mu(y)=\int_0^{+\infty} |\phi(\lambda)|^2\, d\nu_x(\lambda).
        \end{equation}
\end{prop}

\begin{proof}
We first show that the integral in \eqref{kernel_loc_weyl} is well defined for $\Delta$-a.e. $(x,y)\in X\times X$, i.e.
\begin{equation}\label{bound-kernel}
\Delta\text{-a.e. } (x,y)\in X\times X,\qquad \int_0^{+\infty}|\phi(\lambda)|\,d|\Pi_{x,y}|_{\mathrm{HS}}(\lambda)<+\infty.
\end{equation}

Applying \eqref{weighted-TV-bound} with $q(\lambda)=(1+\lambda)^{-\alpha_1-\varepsilon}$, we obtain
\[
\int_0^{+\infty}\frac{d|\Pi_{x,y}|_{\mathrm{HS}}(\lambda)}{(1+\lambda)^{\alpha_1+\varepsilon}}
\leq J_\varepsilon(x)^{1/2}J_\varepsilon(y)^{1/2},\qquad
J_\varepsilon(x):=\int_0^{+\infty}\frac{d\nu_x(\lambda)}{(1+\lambda)^{\alpha_1+\varepsilon}}.
\]
It therefore suffices to prove that $J_\varepsilon(x)$ is finite, uniformly for a.e.\ $x$. Since $\nu_x$ is the Stieltjes measure associated with $E_x(\lambda)=\nu_x([0,\lambda])$, Stieltjes integration by parts gives, with $p=\alpha_1+\varepsilon$,
\begin{equation*}
\int_{[0,+\infty)}(1+\lambda)^{-p}\,d\nu_x(\lambda)=p\int_0^{+\infty}\frac{E_x(\lambda)}{(1+\lambda)^{p+1}}\,d\lambda.
\end{equation*}
By the definition of the pointwise spectral measure, for a.e.\ $x\in X$, every $\lambda\geq0$ and every $t\in(0,1]$,
\begin{equation*}
E_x(\lambda)=\int_{[0,\lambda]}d\nu_x(\xi)
\leq e^{t\lambda}\int_{[0,\lambda]}e^{-t\xi}\,d\nu_x(\xi)
\leq e^{t\lambda}\mathrm{Tr}_{\V_x}(p_t(x,x)).
\end{equation*}

hence, thanks to \eqref{diag_trace_vs_L2} and the bound in Assumption~\ref{item:A1}, we have $E_x(\lambda) \leq C' e^{t\lambda} t^{-\alpha_1}$ (with $C'>0$ independent of $x$, $t$ and $\lambda$). Optimizing with $t=\alpha_1/\lambda$ for large $\lambda$ and enlarging the constant on bounded intervals, we obtain the uniform polynomial bound
\begin{equation}\label{unif-bound-E}
\forall\lambda\geq0,\qquad E_x(\lambda)\leq C_E(1+\lambda)^{\alpha_1}\quad\text{for a.e. }x\in X,
\end{equation}
with $C_E$ independent of $x$. This is the estimate used below; in particular, for the present $\varepsilon>0$, $E_x(\lambda)/(1+\lambda)^{\alpha_1+1+\varepsilon}\leq C_E(1+\lambda)^{-1-\varepsilon}$.
This proves \eqref{bound-kernel} and yields that
$K(x,y)=\int_0^{+\infty}\phi(\lambda)\,d\Pi_{x,y}(\lambda)$
is well-defined for $\Delta$-a.e.\ $(x,y)$ and belongs to $L^\infty_\Delta(X\times X;\Vbox)$.

To show that $K$ is the kernel of $\phi(A)$, define
$K_n(x,y)=\int_0^{+\infty}1_{[0,n]}(\lambda)\phi(\lambda)\,d\Pi_{x,y}(\lambda)$, $n\geq1$.
By Proposition~\ref{prop:spectral-mes}, $K_n$ is the kernel of $\phi(A)1_{[0,n]}(A)$. For $m>n$, the compactly supported identity \eqref{HS_L2_identity_app} gives, for a.e.\ $x$,
\[
\int_X\HSnorm{K_m(x,y)-K_n(x,y)}^2\,d\mu(y)
=\int_{(n,m]}|\phi(\lambda)|^2\,d\nu_x(\lambda).
\]
The right-hand side tends to zero by the polynomial bound on $E_x$ and the assumed decay of $\phi$. Hence $K_n(x,\cdot)$ is Cauchy in $L^2(X;\Hom(\V_x,\V))$. On the other hand, dominated convergence with respect to $|\Pi_{x,y}|_{\mathrm{HS}}$ gives $K_n(x,y)\to K(x,y)$ for $\Delta$-a.e.\ $(x,y)$; extracting an a.e.\ convergent subsequence from the $L^2$ convergence identifies the $L^2$ limit with $K(x,\cdot)$. Therefore, for every $f\in L^2(X;\V)$ and a.e.\ $x$,
\[
\int_XK_n(x,y)f(y)\,d\mu(y)\longrightarrow\int_XK(x,y)f(y)\,d\mu(y).
\]
Spectral calculus also gives $\phi(A)1_{[0,n]}(A)f\to\phi(A)f$ in $L^2(X;\V)$; after extraction of one subsequence, the two pointwise limits agree. Thus $K$ is the kernel of $\phi(A)$. Finally, passing to the limit in the preceding $L^2$ identity yields \eqref{L2rowHS}.

\end{proof}

\begin{remark}\label{rem:consistency_kernel}
The construction is consistent with the heat kernel: for every fixed $t>0$, spectral calculus gives, for $\Delta$-a.e.\ $(x,y)\in X\times X$,
    \begin{equation*}
        p_t(x,y) = \int_0^{+\infty} e^{-t\lambda}\, d\Pi_{x,y}(\lambda) = K_{e^{-tA}}(x,y).
    \end{equation*}
The exceptional set is allowed to depend on $t$; whenever a common representative for all $t>0$ is needed, it is obtained by restricting first to rational times and then using continuity, as in Corollary~\ref{cor:semigroup-column}.
\end{remark}

In the sequel we need to apply $A$ to the first variable of $K_{\phi(A)}$. We therefore use the spectral notation
\begin{equation*}
A K_{\phi(A)}:=K_{(\lambda\phi(\lambda))(A)}\in L^\infty_\Delta(X\times X;\Vbox).
\end{equation*}
The next corollary shows that this notation agrees, columnwise, with the action of the unbounded operator $A$ whenever $\phi$ decays sufficiently fast.

\begin{cor}\label{cor:A-on-kernel}
Let $\phi:[0,+\infty)\to\R$ be a measurable function such that there exist $C,\varepsilon>0$ satisfying
\begin{equation*}
\forall \lambda \geq 0, \quad |\phi(\lambda)| \leq \frac{C}{(1+\lambda)^{\alpha_1 + \varepsilon+1}}.
\end{equation*}
Then $\phi(A)$ and $A\phi(A)$ are integral operators by Proposition~\ref{prop:kernel-weyl}.
Furthermore, for almost every $x\in X$ and every $u\in\V_x$, the section
$y\longmapsto K_{\phi(A)}(y,x)u \in \V_y$
belongs to $D(A)$ and we have in $L^2(X;\V)$
\begin{equation}\label{A_on_column}
A\big(K_{\phi(A)}(\cdot,x)u\big)=K_{A\phi(A)}(\cdot,x)u.
\end{equation}
\end{cor}

\begin{proof}
Let $H=L^2(X;\V)$.
Since $H$ is separable and $A$ is closed, $D(A)$ endowed with the graph norm is separable, hence there exists a countable dense subset $\{f_j\}_{j\ge1}\subset D(A)$.

Fix $j\ge1$.
Since $\phi(A)$ and $A\phi(A)$ are integral operators, there exists a $\mu$-null set $N_j\subset X$ such that for all $x\in X\setminus N_j$ and all $u\in\V_x$,
\begin{align*}
\langle (\phi(A)Af_j)(x),u\rangle_{\V_x}
&=\scalp{Af_j}{K_{\phi(A)}(\cdot,x)u}_{H}, \\
\langle (A\phi(A)f_j)(x),u\rangle_{\V_x}
&=\scalp{f_j}{K_{A\phi(A)}(\cdot,x)u}_{H}.
\end{align*}
By functional calculus, $\phi(A)$ is bounded on $H$ and commutes with $A$ on $D(A)$, hence $\phi(A)Af_j=A\phi(A)f_j$.
Therefore for $x\in X\setminus N_j$ and all $u\in\V_x$,
\begin{equation}\label{adjoint_graph_identity}
\scalp{K_{\phi(A)}(\cdot,x)u}{Af_j}_{H}=\scalp{K_{A\phi(A)}(\cdot,x)u}{f_j}_{H}.
\end{equation}
Now set $N=\bigcup_{j\ge1}N_j$ (still null), and fix $x\in X\setminus N$.
Fix $u\in\V_x$ and define $g=K_{\phi(A)}(\cdot,x)u\in H$ and $h=K_{A\phi(A)}(\cdot,x)u\in H$.
Then \eqref{adjoint_graph_identity} reads
$\forall j\ge1,\ \scalp{g}{Af_j}_{H}=\scalp{h}{f_j}_{H}$.
By density of $\{f_j\}$ in $D(A)$ for the graph norm, this implies
$\forall f\in D(A),\ \scalp{g}{Af}_{H}=\scalp{h}{f}_{H}$.
Since $A$ is self-adjoint, this exactly means $g\in D(A)$ and $Ag=h$, which is \eqref{A_on_column}.
\end{proof}

The same duality argument shows that the heat semigroup acts on the columns of such kernels as expected from spectral calculus; this is used in the proof of Theorem~\ref{thm:main_thm}.

\begin{cor}\label{cor:semigroup-column}
Let $\phi:[0,+\infty)\to\R$ be a measurable function such that there exist $C,\varepsilon>0$ with $|\phi(\lambda)|\leq C(1+\lambda)^{-\alpha_1-\varepsilon}$ for all $\lambda\geq0$. Then, for a.e.\ $x\in X$, every $u\in\V_x$ and every $t>0$,
\begin{equation}\label{semigroup_on_column}
e^{-tA}\big(K_{\phi(A)}(\cdot,x)u\big)=K_{(e^{-t\cdot}\phi)(A)}(\cdot,x)u\qquad\text{in }L^2(X;\V).
\end{equation}
\end{cor}

\begin{proof}
For every $t>0$, the function $\lambda\mapsto e^{-t\lambda}\phi(\lambda)$ satisfies the same decay condition, so both sides of \eqref{semigroup_on_column} belong to $L^2(X;\V)$ by Proposition~\ref{prop:kernel-weyl}. Let $\{f_j\}_{j\geq1}$ be a countable dense subset of $L^2(X;\V)$, and fix $t\in\mathbb{Q}\cap(0,+\infty)$ and $j\geq1$. Applying the kernel identity of Proposition~\ref{prop:kernel-weyl} (in its dual, columnwise form, as in the proof of Corollary~\ref{cor:A-on-kernel}) to the function $e^{-tA}f_j$ for the multiplier $\phi$, and to the function $f_j$ for the multiplier $e^{-t\cdot}\phi$, we obtain a $\mu$-null set $N_{j,t}$ such that, for all $x\notin N_{j,t}$ and $u\in\V_x$,
\begin{align*}
& \scalp{e^{-tA}f_j}{K_{\phi(A)}(\cdot,x)u}_{L^2(X;\V)}=\big\langle(\phi(A)e^{-tA}f_j)(x),u\big\rangle_{\V_x}, \\
& \scalp{f_j}{K_{(e^{-t\cdot}\phi)(A)}(\cdot,x)u}_{L^2(X;\V)}=\big\langle((e^{-t\cdot}\phi)(A)f_j)(x),u\big\rangle_{\V_x}.
\end{align*}
Since $\phi(A)e^{-tA}=(e^{-t\cdot}\phi)(A)$ by spectral calculus and $e^{-tA}$ is self-adjoint, the two right-hand sides coincide, whence
\[
\scalp{f_j}{e^{-tA}\big(K_{\phi(A)}(\cdot,x)u\big)}_{L^2(X;\V)}=\scalp{f_j}{K_{(e^{-t\cdot}\phi)(A)}(\cdot,x)u}_{L^2(X;\V)}
\]
for all $j\geq1$ and all $x\notin N:=\bigcup_{j\geq1,\,t\in\mathbb{Q}_+}N_{j,t}$. By density of $\{f_j\}$, \eqref{semigroup_on_column} holds for all $t\in\mathbb{Q}_+$ and all $x\notin N$. Finally, both sides are continuous in $t\in(0,+\infty)$ with values in $L^2(X;\V)$ (the left-hand side by strong continuity of the semigroup, the right-hand side by \eqref{L2rowHS} and dominated convergence), so \eqref{semigroup_on_column} holds for all $t>0$.
\end{proof}

\subsection{Wave kernel and finite speed of propagation}\label{subsec:wave_kernel}

Here we define the (weak) wave kernel associated to the operator $A$ in our abstract setting, and show that it satisfies the finite speed of propagation property.
\medskip

\subsubsection{Wave kernel existence} For a fixed $s\in\R$, the wave operator $\cos(s\sqrt A)$ need not have an integral kernel in the preceding sense, since the multiplier $\lambda\mapsto\cos(s\sqrt\lambda)$ does not decay at infinity. We therefore introduce a distributional kernel in the time variable.
For example, on a compact Riemannian manifold $(M,g)$, considering the Hodge Laplacian $\Delta_k$ acting on differential $k$-forms, the wave operator has a distributional kernel $K_{\cos(s\sqrt{\Delta_k})}$ given by
\begin{equation*}
K_{\cos(s\sqrt{\Delta_k})}(x,y) = \sum_{n=0}^{+\infty} \cos(s\sqrt{\lambda_n}) \varphi_n(x) \otimes \varphi_n(y)^{\flat},
\end{equation*}
where the equality holds in the sense of $\Lambda^kT^*M \boxtimes (\Lambda^kT^*M)^*$-valued distributions on $M \times M$. Notice that another way to see this kernel is to define $w : M \times M \to \mathcal{S}'(\R) \otimes (\Lambda^kT^*M \boxtimes (\Lambda^kT^*M)^*)$ as follows
\begin{equation*}
    \forall \psi \in \mathcal{S}(\R), \quad \scalp{w(x,y)}{\psi} = \sum_{n=0}^{+\infty} \left(\int_{\R} \psi(s) \cos(s\sqrt{\lambda_n}) ds\right) \varphi_n(x) \otimes \varphi_n(y)^{\flat},
\end{equation*}
where $\scalp{\cdot}{\cdot}$ denotes the duality bracket and in the following all the duality brackets will be denoted similarly to lighten the notation.
Let $f,g \in C^\infty(M;\Lambda^kT^*M)$ and $\psi \in \mathcal{S}(\R)$, then the link between $w$ and the usual wave kernel is given by
\begin{equation*}
    \scalp{\scalp{w}{\psi}}{f\otimes g}=\scalp{\scalp{K_{\cos(\cdot\sqrt{\Delta_k})}}{f\otimes g}}{\psi}.
\end{equation*}
The advantage of the second point of view is that it can be generalized to our abstract setting.

Now, we define a bounded operator $\mathcal{C} : \mathcal{S}(\R) \rightarrow C_b([0,+\infty))$ as follows
\begin{equation}\label{def_Cpsi}
   \forall \lambda\geq 0, \quad \mathcal{C}_\psi(\lambda) = \int_{\R} \psi(s) \cos(s\sqrt{\lambda}) ds.
\end{equation}
Thanks to repeated integration by parts, we have for all $\psi \in \W$ and all $N \geq 0$,
$\mathcal{C}_\psi(\lambda) = O\left(\lambda^{-N}\right)$,
when $\lambda \to +\infty$. We are now able to define the weak wave kernel in our abstract setting.

\begin{prop}\label{prop:weak-wave}
There exists a unique $w \in L^\infty_\Delta\big(X\times X;\W'\otimes (\Vbox)\big)$,
such that for all $\psi \in \W$, for $\Delta$-a.e.\ $(x,y)\in X\times X$,
\begin{equation}\label{weak-wave-ker-def}
\langle w(x,y),\psi\rangle = K_{\mathcal{C}_\psi(A)}(x,y)=\int_0^{+\infty}\mathcal{C}_\psi(\lambda)\,d\Pi_{x,y}(\lambda)\in \Hom(\V_y,\V_x).
\end{equation}
Moreover it satisfies the weak wave equation
\begin{equation}\label{weak_wave_equation}
\forall \psi\in \W,\qquad
A\langle w,\psi\rangle=-\langle w,\psi''\rangle\qquad \Delta\text{-a.e.}
\end{equation}
\end{prop}

\begin{proof}
Let $\psi\in \W$.
Since $\mathcal{C}_\psi(\lambda)=O((1+\lambda)^{-N})$ for all $N$, the function $\mathcal{C}_\psi$ satisfies the decay condition of Proposition~\ref{prop:kernel-weyl}, hence $\mathcal{C}_\psi(A)$ is an integral operator with kernel $K_{\mathcal{C}_\psi(A)}\in L^\infty_\Delta(X\times X;\Vbox)$.
This defines $\langle w,\psi\rangle$ by \eqref{weak-wave-ker-def}.

We next verify that this prescription defines a continuous $\W'$-valued kernel, uniformly in $(x,y)$ in the $L^\infty_\Delta$ sense. For each $k\in\N$, repeated integration by parts in \eqref{def_Cpsi} gives
\begin{equation*}
\sup_{\lambda\ge0}(1+\lambda)^k|\mathcal{C}_\psi(\lambda)| \leq C_k\big(\|\psi\|_{L^1(\R)}+\|\psi^{(2k)}\|_{L^1(\R)}\big).
\end{equation*}
Choose an integer $k>\alpha_1$ and set
$J_k(x):=\int_0^{+\infty}(1+\lambda)^{-k}\,d\nu_x(\lambda)$.
The uniform polynomial estimate \eqref{unif-bound-E}, followed by Stieltjes integration by parts, gives
\[
\mathrm{ess\,sup}_{x\in X}J_k(x)<+\infty.
\]
Applying \eqref{weighted-TV-bound} with $q(\lambda)=(1+\lambda)^{-k}$ therefore yields, for $\Delta$-a.e.\ $(x,y)$,
\begin{multline}
\HSnorm{\langle w(x,y),\psi\rangle}
\leq C_k\big(\|\psi\|_{L^1(\R)}+\|\psi^{(2k)}\|_{L^1(\R)}\big)
\int_0^{+\infty}(1+\lambda)^{-k}\,d|\Pi_{x,y}|_{\mathrm{HS}}(\lambda) \\
\leq C_k'\big(\|\psi\|_{L^1(\R)}+\|\psi^{(2k)}\|_{L^1(\R)}\big).   \label{Cpsi_bound_by_pk}
\end{multline}
The two $L^1$ norms are controlled by standard Schwartz seminorms, for instance by $\|\langle\cdot\rangle^2\psi\|_\infty$ and $\|\langle\cdot\rangle^2\psi^{(2k)}\|_\infty$. Hence $\psi\mapsto\langle w(x,y),\psi\rangle$ is continuous on $\W$, with a bound uniform in the $L^\infty_\Delta$ sense. This proves that $w\in L^\infty_\Delta(X\times X;\W'\otimes\Vbox)$. Uniqueness is immediate from \eqref{weak-wave-ker-def}, since the value of $w$ on every $\psi\in\W$ is prescribed.

Now, for $\lambda\ge0$, integration by parts gives
\begin{equation*}
-\mathcal{C}_{\psi''}(\lambda)=-\int_\R \psi''(s)\cos(s\sqrt{\lambda})\,ds
=\lambda\int_\R \psi(s)\cos(s\sqrt{\lambda})\,ds=\lambda\,\mathcal{C}_\psi(\lambda).
\end{equation*}
Thus, by spectral calculus, $A\mathcal{C}_\psi(A)=-\mathcal{C}_{\psi''}(A)$,
and therefore their kernels satisfy \eqref{weak_wave_equation}.
\end{proof}

\subsubsection{Transmutation formula and finite speed of propagation}
We now establish the Kannai transmutation formula relating the heat and wave kernels; see \cite{Hersh1975MethodTransmutations,Kannai1977OffDiagonal,Romanoff1947OneParameterI,Ungar1971IntegralTransformPrelim} for historical references.

For all $t>0$, we define the function
$\gamma_t(s) = \frac{1}{\sqrt{4\pi t}} e^{-\frac{s^2}{4t}}$,
which is the usual Gaussian heat kernel on the real line. We are now able to state and prove the transmutation formula.

\begin{lem}[Transmutation formula]\label{lem:transmutation}
For every fixed $t>0$, one has
$p_t(x,y)=\langle w(x,y),\gamma_t\rangle$ for $\Delta$-a.e.~$(x,y)\in X\times X$.
\end{lem}

\begin{proof}
    From a Fourier transform computation, we have for all $\lambda \geq 0$ and $t>0$,
    $e^{-t\lambda}  = \mathcal{C}_{\gamma_t}(\lambda)$.
    Thus by Remark~\ref{rem:consistency_kernel} and the definition of $w$, we have for $\Delta$-a.e. $(x,y) \in X \times X,$
        $p_t(x,y) = K_{e^{-tA}}(x,y) = K_{\mathcal{C}_{\gamma_t}(A)}(x,y) = \scalp{w(x,y)}{\gamma_t}$.
\end{proof}

Now, we are able to state and prove the finite speed of propagation property for the wave kernel $w$. It is a direct consequence of Assumption~\ref{item:A2} thanks to the Paley-Wiener theorem. In fact there is an equivalence between Assumption~\ref{item:A2} and the finite speed of propagation property, for more details on this, see \cite{sikora2004riesz}. For all $\delta >0$, set
$\mathcal{D}_\delta =  \{(x,y) \in X \times X \ | \ d(x,y) \leq \delta \}$.
For $K \in L^\infty_\Delta(X \times X;\Vbox)$, we define its support as the essential support (for $\mu\otimes\mu$) of the scalar function $(x,y)\mapsto \HSnorm{K(x,y)}$, and we denote it by $\supp \, K$.
It can be checked that $\supp \, K \subset \mathcal{D}_\delta$ is equivalent to (see \cite{sikora2004riesz}):
for all $f,g \in C_c(X;\V)$ with $\supp \, f \subset B(x,r)$ and $\supp \, g \subset B(y,r')$ such that $d(x,y) > r + r' + \delta$, we have
\begin{equation}\label{charac-supp-ker}
\int_X\int_X \langle K(x',y')f(y'),g(x')\rangle_{\V_{x'}}\,d\mu(x')\,d\mu(y')=0.
\end{equation}
Now we can state and prove the finite speed of propagation property for $w$.

\begin{prop}[Finite speed of propagation of $w$]\label{prop:fsp_wave_ker}
Let $\delta >0$ and $\psi \in C^\infty(\R)$ with $\supp \, \psi \subset [-\delta,\delta]$, then
$\supp \, \langle w,\psi\rangle \subset \mathcal{D}_\delta$.
\end{prop}

\begin{proof}
Let $\delta>0$ and $\psi\in C^\infty(\R)$ with $\supp\, \psi\subset[-\delta,\delta]$.
Let $f,g\in C_c(X;\V)$ with $\supp\, f\subset B(x,r)$ and $\supp\, g\subset B(y,r')$ such that $d(x,y)>r+r'+\delta$.
Then for all $s\in[-\delta,\delta]$, Assumption~\ref{item:A2} implies (via the standard Paley--Wiener argument, see \cite{sikora2004riesz}) that
$\scalp{\cos(s\sqrt{A})f}{g}_{L^2(X;\V)}=0$.
By spectral calculus,
\[
\scalp{\mathcal{C}_\psi(A)f}{g}_{L^2(X;\V)}
=\int_\R \psi(s)\,\scalp{\cos(s\sqrt{A})f}{g}_{L^2(X;\V)}\,ds = 0.
\]
Hence, using the kernel representation of $\mathcal{C}_\psi(A)$,
\begin{equation*}
\int_X\int_X \langle K_{\mathcal{C}_\psi(A)}(x',y')f(y'),g(x')\rangle_{\V_{x'}}\,d\mu(x')\,d\mu(y')=\scalp{\mathcal{C}_\psi(A)f}{g}_{L^2(X;\V)}=0.
\end{equation*}
Using \eqref{charac-supp-ker}, we deduce $\supp\, K_{\mathcal{C}_\psi(A)}\subset\mathcal{D}_\delta$, and since $\langle w,\psi\rangle=K_{\mathcal{C}_\psi(A)}$ by definition, this concludes.
\end{proof}

\begin{remark}\label{rem:finite-speed}
The preceding operator-kernel support statement has the following pointwise consequence: there exists a set $\Omega\subset X\times X$ of full $\mu\otimes\mu$ measure such that, for every $(x,y)\in\Omega$,
\begin{equation}\label{pointwise-wave-support}
\supp w(x,y)\subset\R\setminus(-d(x,y),d(x,y)).
\end{equation}
Indeed, take a rational $\delta>0$ and a countable dense family of test functions in $C_c^\infty((-\delta,\delta))$. Proposition~\ref{prop:fsp_wave_ker}, followed by a countable union of exceptional sets and continuity of $w(x,y)$ on $\W$, shows that $\langle w(x,y),\psi\rangle=0$ whenever $\supp\psi\subset(-\delta,\delta)$ and $d(x,y)>\delta$. Letting rational $\delta\uparrow d(x,y)$ proves \eqref{pointwise-wave-support}. Notice that this conclusion is asserted for ordinary product-a.e.\ pairs; this is what is needed for the off-diagonal integration in Lemma~\ref{resummed-omega}. The $L^\infty_\Delta$ formalism remains essential for diagonal spectral identities, where product-a.e.\ representatives are insufficient. Geometrically, \eqref{pointwise-wave-support} says that a wave travelling at unit speed cannot reach $y$ from $x$ before time $|s|=d(x,y)$.
\end{remark}

\section{Auxiliary estimates}\label{app:tech-lem}
In this Appendix we gather some technical lemmas that we used in the main proofs.

\subsection{Compactly supported distribution estimates}
\begin{lem}\label{lem:localize_weighted_sup_norms}
Let $r>0$ and $m\in\N^*$, and set
$K=\R\setminus (-r,r)=(-\infty,-r]\cup [r,\infty)$.
Let $w\in\mathcal{S}'(\R)$ be a tempered distribution with $\supp \, w\subset K$ and assume that there exists $C>0$ such that
\begin{equation}\label{global_est_w}
\forall \psi\in\mathcal{S}(\R),\qquad
|\langle w,\psi\rangle|
\le C\Big(\sup_{x\in\R}|(1+x^2)\psi(x)|+\sup_{x\in\R}|(1+x^2)\psi^{(2m)}(x)|\Big).
\end{equation}
Then there exists a constant $C_{m}>0$ such that for all $\psi\in\mathcal{S}(\R)$,
\begin{equation}\label{localized_est_w}
|\langle w,\psi\rangle|
\le C_{m}(1+r^{-2m})\Big(\sup_{x\in K}|(1+x^2)\psi(x)|+\sup_{x\in K}|(1+x^2)\psi^{(2m)}(x)|\Big).
\end{equation}
\end{lem}

\begin{proof}
Set $M_k=\sup_{x\in K}|(1+x^2)\psi^{(k)}(x)|$ for $k\in\{0,\ldots,2m\}$. We first claim that
\begin{equation}\label{weighted_intermediate}
M_k\leq L_{m,k}(M_0+M_{2m})\qquad\text{for all }k\in\{0,\ldots,2m\}.
\end{equation}
Indeed, for $n\in\N$ consider $I_n=[r+n,r+n+1]$ and set $f_n(t)=\psi(r+n+t)$ for $t\in[0,1]$. On $I_n$, we have $1+(r+n+t)^2\leq 3(1+(r+n)^2)$ for all $t\in[0,1]$. The Landau--Kolmogorov inequality on $[0,1]$ (see \cite{HardyLittlewoodPolya1952}) gives $\|f_n^{(k)}\|_{L^\infty([0,1])}\leq L'_{m,k}(\|f_n\|_{L^\infty([0,1])}+\|f_n^{(2m)}\|_{L^\infty([0,1])})$. Using the weight bound and $\|f_n\|_\infty\leq M_0/(1+(r+n)^2)$, we get $(1+x^2)|\psi^{(k)}(x)|\leq 3L'_{m,k}(M_0+M_{2m})$ on $I_n$. The same argument applies to $I_n'=[-r-n-1,-r-n]$. Taking the supremum over $n$ gives \eqref{weighted_intermediate}.

Now let $E_+,E_-$ be Seeley extension operators \cite{SeeleyExtension} for $[0,+\infty)$ and $(-\infty,0]$ respectively, satisfying for each $j\in\N$:
\begin{equation}\label{seeley_bounds}
\|(E_+ f)^{(j)}\|_{L^\infty(\R)}\leq A_j\|f^{(j)}\|_{L^\infty([0,+\infty))},\qquad\|(E_- g)^{(j)}\|_{L^\infty(\R)}\leq A_j\|g^{(j)}\|_{L^\infty((-\infty,0])}.
\end{equation}
Define $\tilde\psi_+(x)=(E_+(\psi(\cdot+r)))(x-r)$ (extending $\psi|_{[r,+\infty)}$ to all of $\R$), $\tilde\psi_-(x)=(E_-(\psi(\cdot-r)))(x+r)$ (extending $\psi|_{(-\infty,-r]}$),
and choose $\rho_+,\rho_-\in C^\infty(\R)$ with $\rho_+=1$ for $x\geq 3/4$, $\supp\rho_+\subset[1/2,+\infty)$, and analogously for $\rho_-$. Set $\rho_{\pm,r}(\cdot)=\rho_\pm(\cdot/r)$ and
$ \tilde\psi=\rho_{+,r}\tilde\psi_++\rho_{-,r}\tilde\psi_-$.
Then $\tilde\psi=\psi$ on a neighborhood of $K$, so $\varphi=\psi-\tilde\psi$ has $\supp\varphi\subset(-r,r)$, and since $\supp w\subset K$:
$\langle w,\psi\rangle=\langle w,\tilde\psi\rangle$.

\emph{Zeroth-order bound.} On $K$, $\tilde\psi=\psi$, so $\sup_K|(1+x^2)\tilde\psi|=M_0$. On $[-r,r]$, $1+x^2\leq 1+r^2$, and the Seeley bounds \eqref{seeley_bounds} with $j=0$ give $\|\tilde\psi_\pm\|_\infty\leq A_0\|\psi\|_{L^\infty(\{\pm x\geq r\})}\leq A_0 M_0/(1+r^2)$. Hence $\sup_{|x|\leq r}|(1+x^2)\tilde\psi(x)|\leq 2A_0 M_0$, and combining: $\sup_\R|(1+x^2)\tilde\psi|\leq C_0 M_0$.

\emph{$2m$-th order bound.} By the Leibniz rule:
\[
\tilde\psi^{(2m)}=\sum_{j=0}^{2m}\binom{2m}{j}\rho_{+,r}^{(j)}\tilde\psi_+^{(2m-j)}+\sum_{j=0}^{2m}\binom{2m}{j}\rho_{-,r}^{(j)}\tilde\psi_-^{(2m-j)}.
\]
Using the Seeley bounds, the intermediate estimates \eqref{weighted_intermediate}, and $\|\rho_{\pm,r}^{(j)}\|_\infty\leq r^{-j}\|\rho_\pm^{(j)}\|_\infty$, we get $\sum_{j=0}^{2m}\binom{2m}{j}\|\rho_{\pm,r}^{(j)}\|_\infty\leq K_m(1+r^{-2m})$. Hence
$\|\tilde\psi^{(2m)}\|_{L^\infty([-r,r])}\leq\Gamma_m(1+r^{-2m})\frac{M_0+M_{2m}}{1+r^2}$,
and multiplying by $(1+r^2)$: $\sup_{|x|\leq r}|(1+x^2)\tilde\psi^{(2m)}|\leq\Gamma_m(1+r^{-2m})(M_0+M_{2m})$. On $K$, $\tilde\psi^{(2m)}=\psi^{(2m)}$. Combining: $\sup_\R|(1+x^2)\tilde\psi^{(2m)}|\leq\tilde C_m(1+r^{-2m})(M_0+M_{2m})$.
Applying \eqref{global_est_w} to $\tilde\psi$ gives \eqref{localized_est_w}.
\end{proof}

\subsection{Gaussian volume-doubling integrability}

\begin{lem}\label{lem:gauss-doubling-int}
For every $a\geq0$ and $c>0$ there exists $C=C(a,c,C_D)>0$ such that, for every $x\in X$,
\[
\int_{X}\big(1+d(x,y)\big)^{a}\,e^{-d(x,y)^2/c}\,d\mu(y)\leq C\,\mu(B(x,1))<+\infty.
\]
\end{lem}

\begin{proof}
By the standing assumptions in Subsection~\ref{subsec:hyp}, every ball has finite positive measure. Decompose $X=B(x,1)\cup\bigcup_{k\geq 0}A_k$ with $A_k=\{2^k\leq d(x,y)<2^{k+1}\}$. On $B(x,1)$ the integrand is at most $2^a$. On $A_k$ we have $(1+d(x,y))^{a}\leq 2^{a(k+2)}$, $e^{-d(x,y)^2/c}\leq e^{-4^k/c}$, and $\mu(A_k)\leq\mu(B(x,2^{k+1}))\leq C_D^{k+1}\mu(B(x,1))$ by \eqref{framework-VD}. Hence
\[\int_{X}\big(1+d(x,y)\big)^{a}\,e^{-d(x,y)^2/c}\,d\mu(y)\leq\Big(2^a+\sum_{k=0}^{\infty}2^{a(k+2)}\,C_D^{k+1}\,e^{-4^k/c}\Big)\,\mu(B(x,1))<+\infty,\]
the series converging since $e^{-4^k/c}$ decays super-exponentially in $k$.
\end{proof}

\subsection{From small-time to large-time observability}\label{subsec:small-large}

\begin{lem}\label{lem:small_long_obs}
Let $H$ be a Hilbert space, $B \in \mathcal{L}(H)$, $(S(t))_{t\ge 0}$ be a $C^0$-semigroup on $H$ and
 $h\in C^0_b([0,+\infty))$ positive on $(0,+\infty)$ and nondecreasing. Assume there exist constants $C>0$ and $\tau>0$ such that for every
$T\in(0,\tau]$ and every $f\in H$,
\begin{equation}\label{shortT}
\int_0^T h(t)\,\|S(t)f\|_{H}^2\,dt
\le
C\int_0^T \|BS(t)f\|_{H}^2\,dt.
\end{equation}
Then for every $T\in(0,+\infty]$ and every $f\in H$,
\begin{equation}\label{allT_T}
\int_0^T h(t)\,\|S(t)f\|_{H}^2\,dt
\le
(1+3\|h\|_{\infty} h(\tau/2)^{-1})C\int_0^T \|BS(t)f\|_{H}^2\,dt.
\end{equation}
\end{lem}

\begin{proof}
\emph{Case $\tau=1$.} Let $f\in H$. Since $h(t)\ge h(1/2)$ on $[1/2,1]$, applying \eqref{shortT} with $T=1$ yields
$h(1/2)\int_{1/2}^1 \|S(t)f\|_{H}^2 \, dt \leq \int_0^1 h(t)\|S(t)f\|_{H}^2 \, dt \leq C \int_0^1 \|BS(t)f\|_{H}^2 \, dt$.
For $s>0$, replacing $f$ by $S(s)f$ and using the semigroup property together with $h\le\|h\|_\infty$, we obtain
\begin{equation}\label{trans_bound}
\int_{s+1/2}^{s+1} h(t)\|S(t)f\|_{H}^2 \, dt \leq C' \int_s^{s+1} \|BS(t)f\|_{H}^2 \, dt, \quad C':=\|h\|_{\infty} h(1/2)^{-1} C.
\end{equation}
Let $T>1$ and $N=\floor{2T-2}$. We split $[0,T]=[0,1/2)\cup[1/2,N/2+1)\cup[N/2+1,T]$. On $[0,1/2)$, \eqref{shortT} directly gives the bound $C\int_0^T\|BS(t)f\|_H^2\,dt$. On $[1/2,N/2+1)$, applying \eqref{trans_bound} with $s=k/2$ for $k=0,\dots,N$ and summing yields $\int_{1/2}^{N/2+1} h(t)\|S(t)f\|_{H}^2 \, dt \leq C' \sum_{k=0}^{N}\int_{k/2}^{k/2+1}\|BS(t)f\|_H^2\,dt \leq 2C'\int_0^T\|BS(t)f\|_H^2\,dt$
(the overlap multiplicity of the intervals $[k/2,k/2+1]$ is at most two almost everywhere, and $N/2+1\leq T$).
Finally, since $T-1/2\le N/2+1$ by definition of $N$, \eqref{trans_bound} with $s=T-1$ bounds $\int_{N/2+1}^T$ by $C'\int_0^T\|BS(t)f\|_H^2\,dt$. Summing the three bounds yields \eqref{allT_T} for $\tau=1$ and $T<+\infty$.

\emph{General case.} For $\tau>0$, we apply the above to $\widetilde S(t):=S(\tau t)$ with the weight $\widetilde h(t):=h(\tau t)$; the change of variable $t\mapsto t/\tau$ then yields \eqref{allT_T}. The case $T=+\infty$ follows by monotone convergence.
\end{proof}

\small
\bibliographystyle{abbrv}
\bibliography{biblio}

@article{ErvedozaZuazua2011,
  author  = {Ervedoza, Sylvain and Zuazua, Enrique},
  title   = {Sharp Observability Estimates for Heat Equations},
  journal = {Archive for Rational Mechanics and Analysis},
  year    = {2011},
  volume  = {202},
  pages   = {975--1017},
  doi     = {10.1007/s00205-011-0445-8}
}

@article{CdVHT_AHL_2021,
  author  = {Colin de Verdi\`ere, Yves and Hillairet, Luc and Tr\'elat, Emmanuel},
  title   = {Small-time asymptotics of hypoelliptic heat kernels near the diagonal, nilpotentization and related results},
  journal = {Annales Henri Lebesgue},
  year    = {2021},
  volume  = {4},
  pages   = {897--971},
  doi     = {10.5802/ahl.93},
  url     = {https://ahl.centre-mersenne.org/articles/10.5802/ahl.93/}
}

@incollection{Zuazua2001Carleman,
  author    = {Zuazua, Enrique},
  title     = {Some Results and Open Problems on the Controllability of Linear and Semilinear Heat Equations},
  booktitle = {{C}arleman Estimates and Applications to Uniqueness and Control Theory},
  editor    = {Colombini, Ferruccio and Zuily, Claude},
  series    = {Progress in Nonlinear Differential Equations and Their Applications},
  volume    = {46},
  publisher = {Birkh{\"a}user},
  address   = {Boston, MA},
  pages     = {191--211},
  year      = {2001},
  doi       = {10.1007/978-1-4612-0203-5_14}
}

@article{Miller2004JDE,
  author  = {Miller, Luc},
  title   = {Geometric Bounds on the Growth Rate of Null-Controllability Cost for the Heat Equation in Small Time},
  journal = {Journal of Differential Equations},
  year    = {2004},
  volume  = {204},
  number  = {1},
  pages   = {202--226},
  doi     = {10.1016/j.jde.2004.05.007}
}

@article{Miller2006RLM,
  author  = {Miller, Luc},
  title   = {On Exponential Observability Estimates for the Heat Semigroup with Explicit Rates},
  journal = {Rendiconti Lincei. Matematica e Applicazioni},
  year    = {2006},
  volume  = {17},
  number  = {4},
  pages   = {351--366},
  doi     = {10.4171/RLM/473}
}

@article{Hormander1967Acta,
  author  = {H{\"o}rmander, Lars},
  title   = {Hypoelliptic second order differential equations},
  journal = {Acta Mathematica},
  year    = {1967},
  volume  = {119},
  pages   = {147--171},
  doi     = {10.1007/BF02392081},
  url     = {https://link.springer.com/article/10.1007/BF02392081}
}

@book{AgrachevBarilariBoscain2019,
  author    = {Agrachev, Andrei and Barilari, Davide and Boscain, Ugo},
  title     = {A Comprehensive Introduction to Sub-{R}iemannian Geometry},
  series    = {Cambridge Studies in Advanced Mathematics},
  volume    = {181},
  publisher = {Cambridge University Press},
  address   = {Cambridge},
  year      = {2019},
  doi       = {10.1017/9781108677325},
  isbn      = {9781108476355}
}

@book{FursikovImanuvilov1996,
  author    = {Fursikov, A. V. and Imanuvilov, O. Yu.},
  title     = {Controllability of Evolution Equations},
  series    = {Lecture Notes Series},
  volume    = {34},
  publisher = {Research Institute of Mathematics, Seoul National University},
  address   = {Seoul},
  year      = {1996}
}

@article{Imanuvilov1995Sbornik,
  author  = {Imanuvilov, O. Yu.},
  title   = {Controllability of parabolic equations},
  journal = {Sbornik: Mathematics},
  year    = {1995},
  volume  = {186},
  number  = {6},
  pages   = {879--900}
}

@article{LebeauRobbiano1995CPDE,
  author  = {Lebeau, Gilles and Robbiano, Luc},
  title   = {Contr\^ole exact de l\'equation de la chaleur},
  journal = {Communications in Partial Differential Equations},
  year    = {1995},
  volume  = {20},
  number  = {1--2},
  pages   = {335--356}
}

@article{LeRousseauLebeau2012COCV,
  author  = {Le Rousseau, J\'er\^ome and Lebeau, Gilles},
  title   = {On {C}arleman estimates for elliptic and parabolic operators. Applications to unique continuation and control of parabolic equations},
  journal = {ESAIM: Control, Optimisation and Calculus of Variations},
  year    = {2012},
  volume  = {18},
  number  = {3},
  pages   = {712--747}
}

@article{FernandezCaraZuazua2000ADE,
  author  = {Fern{\'a}ndez-Cara, Enrique and Zuazua, Enrique},
  title   = {The cost of approximate controllability for heat equations: the linear case},
  journal = {Advances in Differential Equations},
  year    = {2000},
  volume  = {5},
  number  = {4--6},
  pages   = {465--514},
  doi     = {10.57262/ade/1356651338},
  url     = {https://projecteuclid.org/journals/advances-in-differential-equations/volume-5/issue-4-6/The-cost-of-approximate-controllability-for-heat-equations--the/ade/1356651338.full}
}

@article{Lions1988SIAMRev,
  author  = {Lions, Jacques-Louis},
  title   = {Exact Controllability, Stabilization and Perturbations for Distributed Systems},
  journal = {SIAM Review},
  year    = {1988},
  volume  = {30},
  number  = {1},
  pages   = {1--68},
  doi     = {10.1137/1030001}
}

@book{Trelat2024CFID,
  author    = {Tr\'elat, Emmanuel},
  title     = {Control in Finite and Infinite Dimension},
  series    = {SpringerBriefs in PDEs and Data Science},
  publisher = {Springer},
  year      = {2024},
  doi       = {10.1007/978-981-97-5948-4},
  isbn      = {978-981-97-5947-7}
}

@article{DoleckiRussell1977SICON,
  author  = {Dolecki, Szymon and Russell, David L.},
  title   = {A general theory of observation and control},
  journal = {SIAM Journal on Control and Optimization},
  year    = {1977},
  volume  = {15},
  number  = {2},
  pages   = {185--220},
  doi     = {10.1137/0315015},
  url     = {https://epubs.siam.org/doi/10.1137/0315015}
}

@article{LaurentLeautaud2021APDE,
  author  = {Laurent, Camille and L\'eautaud, Matthieu},
  title   = {Observability of the heat equation, geometric constants in control theory, and a conjecture of {L}uc {M}iller},
  journal = {Analysis \& PDE},
  year    = {2021},
  volume  = {14},
  number  = {2},
  pages   = {355--423},
  doi     = {10.2140/apde.2021.14.355},
  url     = {https://msp.org/apde/2021/14-2/apde-v14-n2-p02-p.pdf}
}

@article{BeauchardCannarsaGuglielmi2014JEMS,
  author  = {Beauchard, Karine and Cannarsa, Piermarco and Guglielmi, Roberto},
  title   = {Null controllability of {G}rushin-type operators in dimension two},
  journal = {Journal of the European Mathematical Society},
  year    = {2014},
  volume  = {16},
  number  = {1},
  pages   = {67--101},
  doi     = {10.4171/JEMS/428}
}

@article{BeauchardMillerMorancey2015JDE,
  author  = {Beauchard, Karine and Miller, Luc and Morancey, Morgan},
  title   = {2{D} {G}rushin-type equations: Minimal time and null controllable data},
  journal = {Journal of Differential Equations},
  year    = {2015},
  volume  = {259},
  number  = {11},
  pages   = {5813--5845},
  doi     = {10.1016/j.jde.2015.07.007}
}

@article{BeauchardDardeErvedoza2020AIF,
  author  = {Beauchard, Karine and Dard{\'e}, J{\'e}r{\'e}mi and Ervedoza, Sylvain},
  title   = {Minimal time issues for the observability of {G}rushin-type equations},
  journal = {Annales de l'Institut Fourier},
  year    = {2020},
  volume  = {70},
  number  = {1},
  pages   = {247--312},
  doi     = {10.5802/aif.3313}
}

@article{DuprezKoenig2020COCV,
  author  = {Duprez, Michel and Koenig, Armand},
  title   = {Control of the {G}rushin equation: non-rectangular control region and minimal time},
  journal = {ESAIM: Control, Optimisation and Calculus of Variations},
  year    = {2020},
  volume  = {26},
  pages   = {54},
  doi     = {10.1051/cocv/2019001}
}

@article{Hormander1968Acta,
  author  = {H{\"o}rmander, Lars},
  title   = {The spectral function of an elliptic operator},
  journal = {Acta Mathematica},
  year    = {1968},
  volume  = {121},
  pages   = {193--218},
  doi     = {10.1007/BF02391913},
  url     = {https://link.springer.com/article/10.1007/BF02391913}
}

@book{Grigoryan2009_HeatKernel,
  author       = {Alexander Grigor\kern-.15em'{y}an},
  title        = {Heat Kernel and Analysis on Manifolds},
  series       = {AMS/IP Studies in Advanced Mathematics},
  volume       = {47},
  year         = {2009},
  publisher    = {American Mathematical Society / International Press of Boston},
  address      = {Providence, RI / Boston, MA},
  isbn_softcover = {978-0-8218-9393-7},
  isbn_ebook     = {978-1-4704-1750-5},
  pages        = {482},
}

@article{Guichal1985_ControlOperatorHeat,
  author    = {Edgardo N. Güichal},
  title     = {A lower bound of the norm of the control operator for the heat equation},
  journal   = {Journal of Mathematical Analysis and Applications},
  volume    = {110},
  number    = {2},
  pages     = {519--527},
  year      = {1985},
  doi       = {10.1016/0022-247X(85)90313-0}
}

@article{sikora2004riesz,
  author  = {Sikora, Adam},
  title   = {{R}iesz transform, {G}aussian bounds and the method of wave equation},
  journal = {Mathematische Zeitschrift},
  volume  = {247},
  pages   = {643--662},
  year    = {2004},
  doi     = {10.1007/s00209-003-0639-3},
}

@book{BinghamGoldieTeugels1987,
  author    = {Bingham, N. H. and Goldie, C. M. and Teugels, J. L.},
  title     = {Regular Variation},
  series    = {Encyclopedia of Mathematics and its Applications},
  volume    = {27},
  publisher = {Cambridge University Press},
  year      = {1987},
  isbn      = {978-0521307859},
  doi       = {10.1017/CBO9780511721434}
}

@article{DaviesSimon1984,
  author  = {Davies, E. B. and Simon, Barry},
  title   = {Ultracontractivity and the heat kernel for {Schr{\"o}dinger} operators and {Dirichlet} {L}aplacians},
  journal = {Journal of Functional Analysis},
  volume  = {59},
  number  = {2},
  pages   = {335--395},
  year    = {1984},
  doi     = {10.1016/0022-1236(84)90076-4}
}

@book{Davies1989,
  author    = {Davies, E. B.},
  title     = {Heat Kernels and Spectral Theory},
  series    = {Cambridge Tracts in Mathematics},
  volume    = {92},
  publisher = {Cambridge University Press},
  address   = {Cambridge},
  year      = {1989}
}

@incollection{GrigorYan2014HKMMS,
  author    = {Grigor'yan, Alexander and Hu, Jiaxin and Lau, Ka-Sing},
  title     = {Heat kernels on metric measure spaces},
  booktitle = {Geometry and Analysis of Fractals},
  series    = {Springer Proceedings in Mathematics and Statistics},
  volume    = {88},
  pages     = {147--207},
  publisher = {Springer},
  year      = {2014}
}

@article{GrigoryanHu2014DV,
  author       = {Alexander Grigor’yan and Jiaxin Hu},
  title        = {Upper bounds of heat kernels on doubling spaces},
  journal      = {Moscow Mathematical Journal},
  volume       = {14},
  number       = {3},
  pages        = {505--563},
  year         = {2014},
  doi          = {10.17323/1609-4514-2014-14-3-505-563}
}

@article{Romanoff1947OneParameterI,
  author       = {Romanoff, N. P.},
  title        = {On One-Parameter Groups of Linear Transformations. {I}},
  journal      = {Annals of Mathematics},
  series       = {Second Series},
  volume       = {48},
  number       = {2},
  pages        = {216--233},
  year         = {1947},
  doi          = {10.2307/1969167},
  url          = {https://www.jstor.org/stable/1969167}
}

@article{Ungar1971IntegralTransformPrelim,
  author       = {Ungar, A.},
  title        = {On an Integral Transform Related to the Wave and to the Heat Equations},
  journal      = {Notices of the American Mathematical Society},
  volume       = {18},
  pages        = {1100},
  year         = {1971}
}

@incollection{Hersh1975MethodTransmutations,
  author       = {Hersh, Reuben},
  title        = {The Method of Transmutations},
  booktitle    = {Partial Differential Equations and Related Topics},
  editor       = {Goldstein, Jerome A.},
  series       = {Lecture Notes in Mathematics},
  volume       = {446},
  pages        = {264--282},
  publisher    = {Springer},
  address      = {Berlin, Heidelberg},
  year         = {1975},
  doi          = {10.1007/BFb0070606}
}

@article{Kannai1977OffDiagonal,
  author       = {Kannai, Yakar},
  title        = {Off Diagonal Short Time Asymptotics for Fundamental Solutions of Diffusion Equations},
  journal      = {Communications in Partial Differential Equations},
  volume       = {2},
  number       = {8},
  pages        = {781--830},
  year         = {1977},
  doi          = {10.1080/03605307708820048}
}

@article{WangWangZhangZhang2019Observable,
  title   = {Observable set, observability, interpolation inequality and spectral inequality for the heat equation in {$\mathbb{R}^n$}},
  author  = {Wang, Gengsheng and Wang, Ming and Zhang, Can and Zhang, Yubiao},
  journal = {Journal de Math\'ematiques Pures et Appliqu\'ees},
  volume  = {126},
  pages   = {144--194},
  year    = {2019},
  month   = jun,
  doi     = {10.1016/j.matpur.2019.04.009},
  url     = {https://doi.org/10.1016/j.matpur.2019.04.009}
}

@article{EgidiVeselic2018Sharp,
  title   = {Sharp geometric condition for null-controllability of the heat equation on {$\mathbb{R}^d$} and consistent estimates on the control cost},
  author  = {Egidi, Michela and Veseli{\'c}, Ivan},
  journal = {Archiv der Mathematik},
  volume  = {111},
  number  = {1},
  pages   = {85--99},
  year    = {2018},
  month   = jul,
  doi     = {10.1007/s00013-018-1185-x},
  url     = {https://doi.org/10.1007/s00013-018-1185-x}
}

@article{NakicTauferTautenhahnVeselic2020_SharpEstimatesHomogenizationControlCost,
  author       = {Naki\'c, Ivica and T\"aufer, Matthias and Tautenhahn, Martin and Veseli\'c, Ivan},
  title        = {Sharp estimates and homogenization of the control cost of the heat equation on large domains},
  journal      = {ESAIM: Control, Optimisation and Calculus of Variations},
  year         = {2020},
  volume       = {26},
  pages        = {Article no. 54},
  publisher    = {EDP Sciences},
  doi          = {10.1051/cocv/2019058},
  url          = {https://www.numdam.org/articles/10.1051/cocv/2019058/}
}

@article{DuanWangZhang2020,
  author  = {Duan, Yueliang and Wang, Lijuan and Zhang, Can},
  title   = {Observability Inequalities for the Heat Equation with Bounded Potentials on the Whole Space},
  journal = {SIAM Journal on Control and Optimization},
  year    = {2020},
  volume  = {58},
  number  = {4},
  pages   = {1939--1960},
  doi     = {10.1137/19M1296847}
}

@misc{LebeauMoyano2019,
  author        = {Lebeau, Gilles and Moyano, Iv{\'a}n},
  title         = {Spectral Inequalities for the {S}chr{\"o}dinger Operator},
  year          = {2019},
  eprint        = {1901.03513},
  archivePrefix = {arXiv},
  primaryClass  = {math.AP},
  howpublished  = {Preprint, arXiv:1901.03513}
}

@article{NakicTauferTautenhahnVeselic2018,
  author  = {Naki{\'c}, Ivica and T{\"a}ufer, Matthias and Tautenhahn, Martin and Veseli{\'c}, Ivan},
  title   = {Scale-free unique continuation principle for spectral projectors, eigenvalue lifting and {W}egner estimates for random {S}chr{\"o}dinger operators},
  journal = {Analysis \& PDE},
  year    = {2018},
  volume  = {11},
  number  = {4},
  pages   = {1049--1081},
  doi     = {10.2140/apde.2018.11.1049},
  eprint  = {1609.01953},
  archivePrefix = {arXiv},
  primaryClass  = {math.AP}
}

@article{DickeSeelmannVeselic2024,
  author  = {Dicke, Alexander and Seelmann, Albrecht and Veseli{\'c}, Ivan},
  title   = {Spectral inequality with sensor sets of decaying density for {S}chr{\"o}dinger operators with power growth potentials},
  journal = {Partial Differential Equations and Applications},
  year    = {2024},
  volume  = {5},
  pages   = {Article no. 7},
  doi     = {10.1007/s42985-024-00276-0},
  eprint  = {2206.08682},
  archivePrefix = {arXiv},
  primaryClass  = {math.AP}
}

@article{SanchezCalle1984FundamentalSolutions,
  author  = {S{\'a}nchez-Calle, Antonio},
  title   = {Fundamental solutions and geometry of the sum of squares of vector fields},
  journal = {Inventiones Mathematicae},
  year    = {1984},
  volume  = {78},
  number  = {1},
  pages   = {143--160},
  month   = feb,
  doi     = {10.1007/BF01388721},
  url     = {https://doi.org/10.1007/BF01388721}
}

@incollection{Melrose1986PropagationWaveGroupSubelliptic,
  author    = {Melrose, Richard},
  title     = {Propagation for the Wave Group of a Positive Subelliptic Second-Order Differential Operator},
  booktitle = {Hyperbolic Equations and Related Topics: Proceedings of the {T}aniguchi International Symposium, {K}atata and {K}yoto, 1984},
  editor    = {Mizohata, Sigeru},
  publisher = {Academic Press},
  year      = {1986},
  pages     = {181--192},
  isbn      = {978-0-12-501658-2},
  doi       = {10.1016/B978-0-12-501658-2.50015-4}
}

@misc{ColinDeVerdiereHillairetTrelat2022arXiv2212_02920,
  author        = {Colin de Verdi\`ere, Yves and Hillairet, Luc and Tr\'elat, Emmanuel},
  title         = {Spectral asymptotics for sub-{R}iemannian {L}aplacians},
  year          = {2022},
  eprint        = {2212.02920},
  archivePrefix = {arXiv},
  primaryClass  = {math.DG},
  doi           = {10.48550/arXiv.2212.02920},
  url           = {https://arxiv.org/abs/2212.02920},
  howpublished  = {Preprint, arXiv:2212.02920}
}

@article{SeeleyExtension,
   author  = {Seeley, R. T.},
   title   = {Extension of {$C^\infty$} functions defined in a half-space},
   journal = {Proceedings of the American Mathematical Society},
   volume  = {15},
   year    = {1964},
   pages   = {625--626}
}

@article{Ludewig2019StrongShortTime,
  author  = {Ludewig, Matthias},
  title   = {Strong short-time asymptotics and convolution approximation of the heat kernel},
  journal = {Annals of Global Analysis and Geometry},
  year    = {2019},
  volume  = {55},
  pages   = {371--394},
  doi     = {10.1007/s10455-018-9630-4},
  url     = {https://link.springer.com/article/10.1007/s10455-018-9630-4}
}

@article{MazariFouquerPrivatTrelat2025LargeTimeOptimalObservation,
  title   = {Large-time optimal observation domain for linear parabolic systems},
  author  = {Mazari-Fouquer, Idriss and Privat, Yannick and Tr{\'e}lat, Emmanuel},
  journal = {Annales de l'Institut Henri Poincar{\'e} C, Analyse Non Lin{\'e}aire},
  year    = {2025},
  note    = {Published online first (23 Jan 2025)},
  doi     = {10.4171/AIHPC/152},
  eprint  = {2402.03980},
  archivePrefix = {arXiv},
  primaryClass  = {math.AP}
}

@article{LissyZuazua2019InternalObservabilityCoupledSystems,
  title   = {Internal Observability for Coupled Systems of Linear Partial Differential Equations},
  author  = {Lissy, Pierre and Zuazua, Enrique},
  journal = {SIAM Journal on Control and Optimization},
  year    = {2019},
  volume  = {57},
  number  = {2},
  pages   = {832--853},
  doi     = {10.1137/17M1119160}
}

@article{BK24a,
  title        = {A Modified Local {W}eyl Law and Spectral Comparison Results for $\delta'$-Coupling Conditions},
  author       = {Bifulco, Patrizio and Kerner, Joachim},
  journal      = {Journal of Mathematical Physics},
  year         = {2025},
  volume       = {66},
  number       = {3},
  pages        = {033503},
  doi          = {10.1063/5.0239937},
  eprint       = {2407.21719},
  archivePrefix= {arXiv},
  primaryClass = {math.SP},
  note         = {Also available as arXiv:2407.21719},
}

@article{BolteEggerRueckriemen2015HeatKernelResolventMetricGraphs,
  title   = {Heat-kernel and Resolvent Asymptotics for {Schr{\"o}dinger} Operators on Metric Graphs},
  author  = {Bolte, Jens and Egger, Sebastian and Rueckriemen, Ralf},
  journal = {Applied Mathematics Research eXpress},
  year    = {2015},
  volume  = {2015},
  number  = {1},
  pages   = {129--165},
  doi     = {10.1093/amrx/abu009},
  url     = {https://doi.org/10.1093/amrx/abu009}
}

@article{OdzakSceta2019WeylQuantumGraphs,
  title        = {On the {W}eyl Law for Quantum Graphs},
  author       = {Od{\v{z}}ak, Almasa and {\v{S}}{\'c}eta, Lamija},
  journal      = {Bulletin of the Malaysian Mathematical Sciences Society},
  volume       = {42},
  number       = {1},
  pages        = {119--131},
  year         = {2019},
  doi          = {10.1007/s40840-017-0469-9},
  publisher    = {Springer}
}

@misc{AmmariDucaJolyLeBalch2025GGCC,
  title         = {The Graph Geometric Control Condition},
  author        = {Ammari, Ka{\"i}s and Duca, Alessandro and Joly, Romain and Le Balc'h, K{\'e}vin},
  year          = {2025},
  month         = mar,
  eprint        = {2503.18864},
  archivePrefix = {arXiv},
  primaryClass  = {math.OC},
  doi           = {10.48550/arXiv.2503.18864},
  note          = {arXiv:2503.18864},
}

@article{AvdoninEdwardLeugering2023CycleDeltaPrime,
  title   = {Controllability for the wave equation on graph with cycle and delta-prime vertex conditions},
  author  = {Avdonin, Sergei and Edward, Julian and Leugering, G{\"u}nter},
  journal = {Evolution Equations and Control Theory},
  year    = {2023},
  volume  = {12},
  number  = {6},
  pages   = {1542--1558},
  doi     = {10.3934/eect.2023025},
}

@article{ApraizBarcenaPetisco2023ParabolicLoops,
  title   = {Observability and control of parabolic equations on networks with loops},
  author  = {Apraiz, Jone and B{\'a}rcena-Petisco, Jon Asier},
  journal = {Journal of Evolution Equations},
  year    = {2023},
  volume  = {23},
  pages   = {37},
  doi     = {10.1007/s00028-023-00882-2},
}

@book{HardyLittlewoodPolya1952,
  author    = {Hardy, G. H. and Littlewood, J. E. and P{\'o}lya, G.},
  title     = {Inequalities},
  edition   = {2},
  publisher = {Cambridge University Press},
  address   = {Cambridge},
  year      = {1952}
}

@book{deBruijn1958,
  author    = {de Bruijn, N. G.},
  title     = {Asymptotic Methods in Analysis},
  publisher = {North-Holland},
  address   = {Amsterdam},
  year      = {1958}
}

@book{BerkolaikoKuchment2013,
  author    = {Berkolaiko, Gregory and Kuchment, Peter},
  title     = {Introduction to Quantum Graphs},
  series    = {Mathematical Surveys and Monographs},
  volume    = {186},
  publisher = {American Mathematical Society},
  address   = {Providence, RI},
  year      = {2013},
  doi       = {10.1090/surv/186}
}

@article{Varadhan1967,
  author  = {Varadhan, S. R. S.},
  title   = {On the Behavior of the Fundamental Solution of the Heat Equation with Variable Coefficients},
  journal = {Communications on Pure and Applied Mathematics},
  volume  = {20},
  number  = {2},
  pages   = {431--455},
  year    = {1967},
  doi     = {10.1002/cpa.3160200210}
}

@article{CoulhonSikora2008,
  author  = {Coulhon, Thierry and Sikora, Adam},
  title   = {Gaussian Heat Kernel Upper Bounds via the {Phragm\'{e}n--Lindel\"of} Theorem},
  journal = {Proceedings of the London Mathematical Society},
  series  = {3},
  volume  = {96},
  number  = {2},
  pages   = {507--544},
  year    = {2008},
  doi     = {10.1112/plms/pdm050}
}

\end{document}